\documentclass[a4paper,11pt]{amsart}
\usepackage{amsmath,amsthm,amsfonts}
\usepackage[english]{babel}
\usepackage{scalerel}
\usepackage{nicefrac}
\usepackage{diagbox}
\usepackage{tikz-cd}
\usepackage{accents}
\usepackage{amssymb}
\usepackage{mathtools}
\usepackage[latin1]{inputenc}
\usepackage{todonotes}
\usepackage[T1]{fontenc}
\usepackage[shortlabels]{enumitem}
\usepackage{float}
\usepackage{subfigure}
\usepackage{multirow}
\usepackage{mathrsfs}
\usepackage{color}
\usepackage[hidelinks]{hyperref} 
\usepackage[capitalize,nameinlink,noabbrev]{cleveref}
\usepackage{subfigure}
\usepackage[all]{hypcap}    
\newcommand{\CC}{\mathbb{C}}

\newcommand{\ZZ}{\mathbb{Z}}
\newcommand{\eqA}{\mathscr{A}}
\newcommand{\eqK}{\mathscr{K}}

\newcommand{\OO}{\mathcal{O}}

\newcommand{\CCzero}[1]{(\mathbb{C}^{#1},0)}

\newcommand{\CCS}[1]{(\mathbb{C}^{#1},S)}
\newcommand{\GS}[2]{(\mathbb{C}^{#1},S)\rightarrow(\mathbb{C}^{#2},0)}
\newcommand{\Gzero}[2]{(\mathbb{C}^{#1},0)\rightarrow(\mathbb{C}^{#2},0)}
\newcommand{\rank}{\textnormal{rank}}

\newcommand{\Alt}{\textnormal{Alt}}

\def\im{\operatorname{im }}
\def\dim{\operatorname{dim}}
\def\rank{\operatorname{rank}}

\def\im{\operatorname{Im}}

\newcommand{\mdash}{\nobreakdash-\hspace{0pt}}
 \newcommand{\medpar}[1]{\scalerel*{(}{\strut}#1\scalerel{)}{\strut}}

\addto\captionsspanish{}
\addto\captionsspanish{}

\theoremstyle{plain}
\newtheorem{theorem}{Theorem}[section]
\newtheorem{lemma}[theorem]{Lemma}
\newtheorem{corollary}[theorem]{Corollary}
\newtheorem{proposition}[theorem]{Proposition}

\theoremstyle{definition}
\newtheorem{definition}[theorem]{Definition}
\newtheorem{conjecture}[theorem]{Conjecture}
\newtheorem{example}[theorem]{Example}

\theoremstyle{remark}
\newtheorem*{remark}{Remark}
\newtheorem*{note}{Note}

\begin{document}

\author{R. Giménez Conejero and
J.J.~Nu\~no-Ballesteros}

\title[Singularities of mappings on ICIS]
{Singularities of mappings on ICIS  and applications to Whitney equisingularity}

\address{
Alfr\'ed R\'enyi Institute of Mathematics, Re\'altanoda utca 13-15,
H-1053 Budapest, 
Hungary
}
\email{Roberto.Gimenez@uv.es}
\address{Departament de Matem\`atiques,
Universitat de Val\`encia, Campus de Burjassot, 46100 Burjassot
SPAIN. \newline Departamento de Matemática, Universidade Federal da Paraíba
		CEP 58051-900, João Pessoa - PB, BRAZIL}
\email{Juan.Nuno@uv.es}

\thanks{The first named author has been partially supported by MCIU Grant FPU16/03844. Grant PGC2018-094889-B-100 funded by MCIN/AEI/ 10.13039/501100011033 and by ``ERDF A way of making Europe''.}

\begin{abstract} 
We study germs of analytic maps $f:(X,S)\to(\CC^p,0)$, when $X$ is an \textsc{icis} of dimension $n<p$. We define an image Milnor number, generalizing Mond's definition, $\mu_I(X,f)$ and give results known for the smooth case such as the conservation of this quantity by deformations. We also use this to characterise the Whitney equisingularity of families of corank one map germs $f_t\colon(\CC^n,S)\to(\CC^{n+1},0)$ with isolated instabilities in terms of the constancy of the $\mu_I^*$-sequences of $f_t$ and the projections $\pi\colon D^2(f_t)\to\CC^n$, where $D^2(f_t)$ is the \textsc{icis} given by double point space of $f_t$ in $\CC^n\times\CC^n$. The $\mu_I^*$-sequence of a map germ consist of the image Milnor number of the map germ and all its successive transverse slices.
\end{abstract}

\maketitle
\section{Introduction}

The singularities of complex analytic mappings $f\colon X\to Y$ are well understood in general when $X$ and $Y$ are complex manifolds.
They are described by the $\eqA$-equivalence classes of holomorphic map germs $f\colon(\CC^n,S)\to(\CC^{p},0)$. On one hand, the classical Thom-Mather theory provides infinitesimal methods to characterize notions such as stability, finite determinacy, versality, etc. in terms of some algebraic invariants. On the other hand, a more recent approach pioneered by Mond, Damon, Gaffney, among others, is based on techniques of  deformation theory, in a similar way as it is done in the case of singularities of complex analytic spaces. A modern reference for both theories is the recent book \cite{Mond-Nuno2020}.

Basically, if $f\colon(\CC^n,S)\to(\CC^{p},0)$ is $\eqA$-finite, then it has isolated instability, by the Mather-Gaffney criterion. If, in addition, $(n,p)$ are nice dimensions or $f$ has corank one, then we can take a stable perturbation $f_s$ which plays the role of the Milnor fiber in the case of a hypersurface with isolated singularity. The analytic and topological invariants of $f_s$ are in fact invariants of the map germ $f$ which include, for instance, the discriminant Milnor number $\mu_\Delta(f)$ when $n\ge p$ or the image Milnor number $\mu_I(f)$ when $p=n+1$. 

A natural question is what happens when we consider mappings $f\colon X\to Y$ and allow $X$ and $Y$ to have also singularities themselves. Since we work locally, we can assume that $(Y,y)$ is embedded in $(\CC^p,0)$, so we consider the case of map germs $f\colon(X,S)\to(\CC^{p},0)$, where $X$ is a complex analytic space. In \cite{Mond1994} Mond and Montaldi developed the Thom-Mather theory 
of pairs $(X,f)$, where $X$ is a complete intersection with isolated singularities (\textsc{icis}) and $f\colon(X,S)\to(\CC^{p},0)$ is a complex analytic map germ.
They also studied deformations of such pairs, mainly in the case that $n=\dim X\ge p$. In particular, they showed that if $(n,p)$ are nice dimensions, then the discriminant Milnor number $\mu_\Delta(X,f)$ is greater than or equal to the $\eqA_e$-codimension, with equality in the weighted homogeneous case. This extended a well known theorem by Damon-Mond in the case that $X$ is smooth (see \cite{Damon1991a}).

In this paper, we are interested in deformations of pairs $(X,f)$, with $n=\dim X< p$, mainly in the case $p=n+1$. When $(n,p)$ are nice dimensions or when $f$ has corank one, a stable perturbation of $(X,f)$ is a pair $(X_s,f_s)$, where $X_s$ is a smoothing of $X$ and $f_s\colon X_s\to\CC^p$ is stable in the usual sense. If $p=n+1$, the image $f_s(X_s)$ has the homotopy type of a wedge of spheres and the number of such spheres is the image Milnor number $\mu_I(X,f)$. A celebrated conjecture by Mond says that $\mu_I(X,f)$ is greater than or equal to the $\eqA_e$-codimension, with equality in the weighted homogeneous case. The conjecture was originally stated by Mond in \cite{Mond1991} when $X$ is smooth and is known to be true when $X$ is smooth and $n=1,2$ (see \cite{Mond1991,Mond1995}) or when $X$ is a plane curve (see \cite{Ament2017}). But in all other cases the conjecture remains open.

When $p>n$, the homology of the image of a stable perturbation $f_s$ can be described in terms of the alternating homology of the multiple point spaces, by means of the image computing spectral sequence \textsc{icss}. This technique was introduced by Goryunov and Mond in \cite{Goryunov1993} in the case that $X$ is smooth and has been developed later by other authors (see \cite{CisnerosMolina2019,Goryunov1995,Houston2007}). Here we show how to adapt the \textsc{icss} to the case that $X$ is an \textsc{icis}. Regarding this, we also study the case of germs of any corank in \cref{sec:3}.

Our original motivation to study deformations of pairs $(X,f)$ is because we were interested in the Whitney equisingularity (\textsc{we}) of families of map germs $f_t\colon(\CC^n,0)\to(\CC^p,0)$. In \cite{Gaffney1993}, Gaffney proved a very general theorem: the family is \textsc{we} if, and only if, it is excellent (in Gaffney's sense) and all the polar multiplicities in the source and target of $f_t$ are constant on $t$. There are two problematic points in this theorem. The first one is that for each $d$-dimensional stratum of the stratification by stable types in the source, or target, we need $d+1$ invariants, so the total number of invariants we need to control the \textsc{we} is huge. Hence, the main question is to find a minimal set of invariants whose constancy is equivalent to the \textsc{we}.

The second problem is that, in general, it is not easy to check whether the family $f_t$ is excellent or not. Roughly speaking, $f_t$ is excellent if
there is no coalescence of instabilities nor $0$-stable type singularities in the family. The crucial point here is to find a numerical invariant (associated to each member $f_t$ of the family), whose constancy implies that the family is excellent.

In a recent paper, \cite{GimenezConejero2021}, we solved the second question for families of multi-germs $f_t\colon(\CC^n,S)\to(\CC^{n+1},0)$ of corank one. We prove that the family is excellent if the image Milnor number $\mu_I(f_t)$ is constant on $t$, which solved a conjecture posed by Houston in these dimensions (see \cite[Conjecture 6.2]{Houston2010}).
With respect to the first question, some partial results can be found in the papers by 
Jorge Pérez and Saia (see \cite{JorgeSaia2006}) or Houston (see \cite{Houston2011}).

In the last part of this paper we provide a minimal set of invariants which control the \textsc{we} of families $f_t\colon(\CC^n,S)\to(\CC^{n+1},0)$ of corank one. We follow the approach of Teissier in \cite{Teissier1982} for hypersurfaces with isolated singularity or Gaffney in \cite{Gaffney1993} for {\sc icis}, based on the $\mu^*$-sequence. In our case, the \textsc{we} in the target is controlled by the $\mu_I^*$-sequence of $f$, obtained by taking successive transverse slices of the map germ $f$. However, in order to control the \textsc{we} in the source, we make use of the $\mu_I^*$-sequence of $\pi\colon D^2(f)\to(\CC^n,S)$, where $D^2(f)$ is the double point space of $f$ in $\CC^n\times\CC^n$ and $\pi$ is the projection onto the first component. We note that, since $f$ has corank one, $D^2(f)$ is an {\sc icis} of dimension $n-1$ and $\pi$ has isolated instabilities. For this reason, we need to develop the deformation theory of pairs $(X,f)$, where $X$ is an {\sc icis}.

Our main result about \textsc{we} is presented in Section 6 (see \cref{equisingularidad}). We show that a family $f_t\colon(\CC^n,S)\to(\CC^{n+1},0)$ of corank one is \textsc{we} if, and only if, the sequences $\mu_I^*(f_t)$ and $\tilde{\mu}_I^*\medpar{D^2(f_t),\pi}$ are constant on $t$. For a pair $(X,f)$, the sequence 
$\mu_I^*(X,f)$ is the sequence of image Milnor numbers of the successive transverse slices of $(X,f)$, starting with $\mu_I(X,f)$ (see \cref{sequences}) and we use the notation $\tilde{\mu}_I^*(X,f)$ when we omit the first term $\mu_I(X,f)$ in the sequence. As a byproduct, we also deduce that $f_t$ is \textsc{we} in the target if, and only if, the sequence $\mu_I^*(f_t)$ is constant on $t$.

To prove this result, we also prove in \cref{sec:Le-Greuel} a version of a Lê-Greuel formula for map germs on \textsc{icis}, which generalizes a result of \cite{Nuno-Ballesteros2019}. This gives the link between Gaffney's polar invariants, the image Milnor numbers and $\mu_I^*$-sequence.
Furthermore, we define a new invariant in the source, the double point Milnor number $\mu_D$, that coincides with the image Milnor number of the projection $\pi:D^2(f)\to\CC^n$ when $f$ has corank one (see \cref{Df}).
\newline

\emph{Acknowledgements:}
The authors thank the anonymous referee for the careful reading and valuable suggestions.

%
%
%

\section{Map germs with an \textsc{icis} in the source}

In \cite{Mond1994} Mond and Montaldi developed the Thom-Mather theory 
of singularities of mappings defined on an isolated complete intersection singularity (\textsc{icis}). They also extended Damon's results in \cite{Damon1991a}, which related the $\eqA_e$-versal unfolding of a map germ $f$ with the $\eqK_{D(G)}$-versal unfoldings of an associated map germ which induces $f$ from a stable map $G$. In particular, when the target has greater dimension than the source or both dimensions coincide, they proved that the discriminant Milnor number $\mu_\Delta(X,f)$ is greater than or equal to the $\eqA_e$-codimension, with equality in the weighted homogeneous case. This is a generalisation of a theorem of Damon and Mond in the case of mappings between smooth manifolds (cf. \cite{Damon1991}). Here, we study what happens when the dimension of the source is one less than the dimension of the target and we consider the image Milnor number $\mu_I(f)$ instead of $\mu_\Delta(f)$.
\newline

First of all, we fix a bit of notation to get rid of some details. Along the text $(X,S)$ will be a multi-germ of \textsc{icis} and $f\colon(X,S)\rightarrow(\CC^p,0)$ will be a holomorphic map germ, written also as $(X,f)$. This kind of germs are called \textit{germs on \textsc{icis}} as well, and we may omit the base set of the germ if it does not provide relevant information or it is clear from the context.

\begin{definition}[see {\cite[p. 4]{Mond1994}}]\label{def:critset}
We will say that $x\in X$ is a \textit{critical point} of $(X,f)$ if either $X$ is smooth at $x$ and $f$ is not submersive at $x$ or if $x$ is a singular point of $X$. Besides, we will denote the set of critical points by $\Sigma(X,f)$, in a similar fashion as in the case of a smooth source. Furthermore, we will say that $(X,f)$ has \textit{finite singularity type} if the restriction of $f$ to $\Sigma(X,f)$ is finite-to-one.
\end{definition}



To be precise, we are not interested in map germs of the form $f:(X,S)\rightarrow (\CC^p,0)$, but in equivalence classes of these maps by a certain relation. This is very common in the study of map germs, and in this text we focus on its analytic structure or, in other words, we study map germs modulo change of coordinates in source and target.

\begin{definition}\label{aeq}
Two map germs $f,g:(X,S)\rightarrow (\CC^p,0)$ are $\eqA$\textit{-equivalent} if there are germs of biholomorphisms $\phi$ of $(X,S)$ and $\psi$ of $(\CC^p,0)$ such that the following diagram is commutative:
$$\begin{tikzcd}
 (X,S) \arrow[r, "f" ]\arrow[d, "\sim"  {anchor= north, rotate=90, inner sep=.6mm}," \phi"']& (\CC^p,0)\arrow[d, " \psi","\sim"'  {anchor= south, rotate=90, inner sep=.3mm}] \\
 (X,S) \arrow[r, "g" ]& (\CC^p,0)
\end{tikzcd}.$$

\end{definition}

There is developed a huge mathematical machinery concerning $\eqA$\mdash  equivalence, also known as Thom-Mather theory (see, for example, \cite{Mond-Nuno2020} or \cite{Wall1981}). For example, to study the singularities that appear \textsl{near to a given germ} (in the sense of perturbations of map germs) we mainly use unfoldings. We recall the following definitions that generalize the smooth case.

\begin{definition}[cf. {\cite[Definition 1]{Mond1994}}] Let $f\colon\left(X,S\right)\rightarrow\left(\CC^p,0\right)$.
\begin{enumerate}[label={(\roman*)},font=\itshape]
	\item An \textit{unfolding of the pair }$\left(X,f\right)$ \textit{over a smooth space germ} $(W,0)$ is a map germ $F\colon(\mathcal{X},S')\rightarrow \left(\CC^p\times W,0\right)$ together with a flat projection $\pi\colon(\mathcal{X},S')\rightarrow (W,0)$ and an isomorphism $j\colon(X,S)\to\big(\pi^{-1}(0),S'\big)$  such that the following diagram commutes
$$ \begin{tikzcd}[column sep=tiny]
&(X,S) \arrow[dr, "f\times\left\{0\right\}" ] \arrow[dl,  "j" ' ]&  \\
\medpar{\pi^{-1}(0),j(S)} \arrow[d,hook] & &\big(\CC^p\times\left\{0\right\}, 0\big) \arrow[d,hook]\\
\medpar{\mathcal{X},j(S)}\arrow[rr, "F"] \arrow[dr,"\pi" '] && (\CC^p\times W,0) \arrow[dl,"\pi_2"]\\
&(W,0)&
\end{tikzcd} ,$$
where $\pi_2:\CC^p\times W\rightarrow W$ is the Cartesian projection. In this case, $W$ is the \textit{parameter space of the unfolding}, and in general we use $(\CC^d,0)$ instead of $(W,0)$. In short, we will use also $(\mathcal{X},\pi,F,j)$ to denote the unfolding. 

\item Given an unfolding $(\mathcal{X},\pi,F,j)$ of $(X,f)$, the map $f_t:X_t\rightarrow \CC^p$ induced from $F$ on $X_t\coloneqq \pi^{-1}(t)$ is called the \textit{perturbation of }$(X,f)$\textit{ induced by the unfolding}, and is abbreviated to the pair $(X_t,f_t)$.

\item In this context, an \textit{unfolding of $f$} is an unfolding of $(X,f)$ with $\mathcal{X}=X\times\CC^d$ and with $\pi:\mathcal{X}\rightarrow \CC^d$ the Cartesian projection. This coincides with the usual definition for smooth spaces.

\item Two unfoldings $(\mathcal{X},\pi,F,j)$ and $(\mathcal{X}',\pi',F',j')$ over $W$ are \textit{isomorphic} if there are isomorphisms $\Phi:\mathcal{X}\rightarrow\mathcal{X}'$ and $ \Psi: \CC^p\times \CC^d \rightarrow \CC^p \times \CC^d$ such that $\Psi$ is an unfolding of the identity over $\CC^d$ and the following diagram commutes:
$$\begin{tikzcd}[column sep=0.4cm]
 & \medpar{\mathcal{X},j(S)} \arrow[dd," \Phi"',"\sim"  {anchor= north, rotate=90, inner sep=.6mm}] \arrow[rr, "F"]\arrow[dr, "\pi"]& & (\CC^p\times \CC^d,0)\arrow[dd," \Psi","\sim"'  {anchor= south, rotate=90, inner sep=.3mm}] \arrow[dl, "\pi_2"']\\
(X,S) \arrow[ur, "j"] \arrow[dr, "j'"'] & & (\CC^d,0) & \\
 & \medpar{\mathcal{X}',j'(S)} \arrow[rr, "F'"] \arrow[ur, "\pi'"]& & (\CC^p\times \CC^d,0)\arrow[ul, "\pi_2"']\\
\end{tikzcd}.$$

\item If $(\mathcal{X},\pi,F,j)$ is an unfolding of $(X,f)$ over $(\CC^d,0)$, a germ $\rho: (\CC^{r},0) \rightarrow (\CC^d,0)$ induces and unfolding $(\mathcal{X}_\rho,\pi_\rho,F_\rho,j_\rho)$ of $(X,f)$ by a \textit{base change} or, in other words, by the fibre product of $F$ and $\text{id}_{\CC^p}\times \rho$:
$$\begin{tikzcd}
\mathcal{X}_\rho\coloneqq \mathcal{X}\times_{\CC^p\times \CC^d}\left(\CC^p\times \CC^s\right)\arrow[r,"F_\rho"]\arrow[d]&\CC^p\times \CC^s\arrow[d,"\text{id}_{\CC^p}\times \rho"]\\
\mathcal{X}\arrow[r,"F"]&\CC^p\times \CC^d
\end{tikzcd},$$
where we omit the points of the germs for simplicity. 

\item The unfolding $(\mathcal{X},\pi,F,j)$ is \textit{versal} if every other unfolding, for example $(\mathcal{X}',\pi',F',j')$, is isomorphic to an unfolding induced from the former by a base change, $(\mathcal{X}_\rho, \pi_\rho, F_\rho, j_\rho)$. A versal unfolding is called \textit{miniversal} if it has a parameter space with minimal dimension.
\end{enumerate}
\end{definition}

As we were saying, an unfolding shows information about the deformations of the germ it unfolds. Roughly speaking, a versal unfolding shows all the possible information an unfolding can show, and if it is miniversal then the information of the unfolding is found without redundancies. On the other hand, there is a type of unfolding that does not contain relevant information: a trivial unfolding.

\begin{definition}\label{stable}
A \textit{trivial unfolding} of a map germ $f$ is an unfolding that is isomorphic to the constant unfolding $(X\times\CC^d,\pi_2,f\times\textnormal{id}_{\CC^d},i)$, 
$$ \begin{tikzcd}[column sep=tiny]
&(X,S) \arrow[dr, "f\times\left\{0\right\}" ] \arrow[dl, hook, "i" ' ]&  \\
\big(X\times\CC^d,S\times\{0\}\big)\arrow[rr, "f\times\textnormal{id}_{\CC^d}"] \arrow[dr,"\pi_2" '] && (\CC^p\times \CC^d,0) \arrow[dl,"\pi_2"]\\
&(\CC^d,0)&
\end{tikzcd} ,$$
where $\pi_2$ is the projection on the second factor and $i$ is the inclusion $(X,S)\hookrightarrow \big(X\times\CC^d,S\times\{0\}\big)$.

On this regard, a map germ is \textit{stable} if every unfolding is trivial. If the map germ is not stable, we say that it has an \textit{instability} or that it is \textit{unstable}.
\end{definition}

Following the idea that an unfolding shows information about the perturbations of a germ, it is evident from the definition of stability of a map germ that a map germ is stable if, and only if, it is its own miniversal unfolding. Another way of seeing this is that every deformation is $\eqA$\mdash equivalent to the original map germ if, and only if, it is stable. As a consequence of this, we see that if a map germ is stable then $(X,S)$ is smooth and $f$ is stable in the usual sense (see \cite[Definition 3.4]{Mond-Nuno2020}). This is related to \cref{sumaAes} below, where the Tjurina number of $(X,0)$ appears.

The concept of \textit{stabilisation} is of high interest in the study of singularities of map germs (for example, see \cite{Damon1991,Mond1991}). Roughly speaking, it is a one parameter unfolding such that every instability lies in the fiber of zero:

\begin{definition}[cf. {\cite[Definition 2]{Mond1994}}]
A stabilisation of a map germ $f:(X,S)\to(\CC^p,0)$ is an unfolding $(\mathcal{X},\pi,F,j)$ such that the parameter space has dimension one and $f_s: X_s\to \CC^p$ has only stable singularities for $s\neq0$, where $f_s$ is the map induced by $F$. 
\end{definition}

How far the map germ $(X,f)$ is of being stable is measured by means of its $\eqA$-codimension:

\begin{definition}
Let $(X,S)\subset(\CC^N,S)$ be a germ of an \textsc{icis} and consider a map germ $f:(X,S)\rightarrow(\CC^p,0)$. We define the following objects:
\begin{enumerate}[label={(\roman*)},font=\itshape]
\item $\theta_{\CC^p,0}$ is the module of vector fields on $(\CC^p,0)$,

\item the module of tangent vector fields defined on $(X,S)$ is
$$\theta_{X,S}\coloneqq\frac{\text{Der}\big(\text{-log}(X,S)\big)}{I(X)\text{Der}\big(\text{-log}(X,S)\big)},$$
where $\text{Der}\big(\text{-log}(X,S)\big)$ are the vector fields on $(\CC^N,S)$ tangent to $(X,S)$,

\item $\theta(f)$ is the module of vector fields along $f$,

\item $\omega f:\theta_{\CC^p,0}\rightarrow \theta(f)$ is the composition with $f$, and

\item $tf:\theta_{X,S}\rightarrow \theta(f)$ is the composition with the differential of a smooth extension of $f$.
\end{enumerate}

Then, the $\mathcal{O}_{\CC^p,0}$-module
$$ N\eqA_e f \coloneqq \frac{\theta(f)}{tf(\theta_{X,S})+\omega f(\theta_{\CC^p,0})} $$
is the $\eqA_e$-\textit{normal space} and its dimension as vector space is the $\eqA_e$\textit{-codimension of} $f$. As usual, we will say that $f$ is $\eqA$-finite if this dimension is finite.

In contrast, the $\eqA_e$\textit{-codimension of the pair }$(X,f)$ is the dimension of the parameter space of a miniversal unfolding of the pair $(X,f)$, if it exists, and it is infinite otherwise. If the $\eqA$-codimension of $(X,f)$ is finite, we say that $(X,f)$ is $\eqA$\mdash finite. 
%
\end{definition}

It is reasonable to ask for the relation between the $\eqA_e$-codimension of $(X,f)$, the $\eqA_e$-codimension of $f$ and the Tjurina number of $X$. This is addressed in \cite[Theorem 1.4]{Mond1994} for the case of mono-germs, i.e., when $S$ is a point:

\begin{theorem}\label{sumaAes}
Let $(X,0)$ be an \textsc{icis} and $f:(X,0)\rightarrow (\CC^p,0)$ of finite singularity type, then $f$ is $\eqA$\mdash finite if, and only if, $(X,f)$ is $\eqA$\mdash finite. Furthermore, in this case,
$$\eqA_e\textnormal{-codim}(X,f)=\eqA_e\textnormal{-codim}(f)+\tau(X,0).$$
\end{theorem}

\begin{note}
With this result we see clearly that if a map germ has smooth source and it is stable in the usual sense then it is stable (and vice versa).
\end{note}

This theorem allows us to prove a very useful result, the \textit{Mather-Gaffney criterion} for germs with an \textsc{icis} in the source.

\begin{proposition}\label{MGcriterion}
A map germ $f:(X,S)\rightarrow(\CC^p,0)$ is $\eqA$\mdash finite if, and only if, it has isolated instability.
\end{proposition}
\begin{proof}
The instabilities could come from points where $X$ is not smooth, which are isolated. On the other hand, on the smooth points we have the usual Mather-Gaffney criterion (see, for example, \cite[Theorem 4.5]{Mond-Nuno2020}), therefore, these points are isolated as well. Using \cref{sumaAes} finishes the proof.\qed\newline
\end{proof}

As usual, we would like to relate those algebraic $\eqA$\mdash invariants with some invariants with a topological flavour. We study what happens when the dimension of the target is greater than the dimension of the source, especially when the difference is $1$. The image Milnor number $\mu_I(f)$ of $f\colon\CCS{n}\rightarrow\CCzero{n+1}$ was introduced by Mond: the image of a stable perturbation has the homotopy type of a wedge of $n$-spheres and the number of such spheres is $\mu_I(f)$ (see \cite[Theorem 1.4]{Mond1991}). Of course, one expects that, in the case where the source is an \textsc{icis}, the situation is similar. There is a hypothesis that simplifies many arguments because it gives extra structure to the objects we study, the corank one hypothesis:

\begin{definition}
We say that $f:(X,S)\rightarrow(\CC^p,0)$ has \textit{corank }$r$ if it has a smooth extension of corank $r$.
\end{definition}

Indeed, Goryunov proved that the image of a stable perturbation of $f:(X,0)\rightarrow (\CC^p,0)$ has non-trivial homology only in certain degrees if $n<p$ and $f$ has corank one, see \cite[Theorem 3.3.1]{Goryunov1995}. Furthermore, in \cite[p. 13]{Mond1994}, Mond and Montaldi proved that, for a map germ $f:(X,0)\to(\CC^p,0)$,  the discriminant locus of a stable perturbation has the homotopy type of a wedge of spheres if $\dim X=n\geq p$, but the same proof works when $p=n+1$ . For the sake of completeness and because we will take this technique a bit further, we outline a proof when $p=n+1$ based on the result of Siersma showed below (see also the previous work of L\^{e} D{\~{u}}ng Tr\'{a}ng in \cite{Trang1987,Trang1992}).

Consider first $g:\CCzero{N+1}\rightarrow\CCzero{}$ and allow isolated and non-isolated singularities. Recall the Milnor fibration: for $\epsilon,\eta$ small enough
$$ g:g^{-1}(D_\eta)\cap B_\epsilon\rightarrow D_\eta$$
 is a locally trivial fibre bundle over $D_\eta\setminus\left\{0\right\}$. Also, consider an unfolding of $g$, $G:(\CC^{N+1}\times\CC^{r},0)\rightarrow(\CC\times\CC^{r},0)$ with $G(x,u)=\left(g_u(x),u\right)$.

\begin{definition}[see {\cite[p. 2]{Siersma1991}}]\label{deftopsiersma}
We say that the unfolding $G$ is \emph{topologically trivial over the Milnor sphere $S_\epsilon$} if, for $\eta$ and $\rho$ small enough,
\begin{equation}\label{top-triv-sphere}
\begin{tikzcd}
(S_\epsilon \times \mathring{B}_\rho) \cap G^{-1}(\mathring{D}_\eta) \arrow[r, "G"]& \mathring{D}_\eta\times \mathring{B}_\rho 
\end{tikzcd}
\end{equation}
is a stratified submersion with strata $\left\{0\right\}\times \mathring{B}_\rho$ and $(\mathring{D}_\eta \setminus \left\{0\right\})\times \mathring{B}_\rho$ on $\mathring{D}_\eta\times \mathring{B}_\rho $ and the induced stratification on $(S_\epsilon \times \mathring{B}_\rho) \cap G^{-1}(\mathring{D}_\eta)$.
\end{definition}

\begin{theorem}[cf. {\cite[Theorem 2.3]{Siersma1991}}]\label{siersma}
With the notation of \cref{deftopsiersma}, let $G$ be a deformation of $g$ which is topologically trivial over a Milnor sphere. Let $u\in \mathring{B}_\rho$ and suppose that all the fibres of $g_u$ are smooth or have isolated singularities except for one special fibre $X_{u}\coloneqq g_u^{-1}(0)\cap B_\epsilon$. Then $X_u$ is homotopy equivalent to a wedge of spheres of dimension $n$ and its number is the sum of the Milnor numbers over all the fibres different from $X_u$.
\end{theorem}

Finally, we give the main $\eqA$\mdash invariant regarding this kind of germs, the image Milnor number, as Mond did in \cite{Mond1991}. \cref{muICIS} below justifies the definition, which is nothing more than our version of \cite[Theorem 2.4]{Mond1994} (see also \cite[pp. 219--220]{Damon1991}). 

\begin{proposition}\label{muICIS}
Let $f\colon(X,S)\rightarrow (\CC^{n+1},0)$ be $\eqA$-finite, where $X$ is an \textsc{icis} with $\dim(X)=n$. Suppose also that $f$ has corank one or $(n,n+1)$ are nice dimensions in the sense of Mather. In this case, if  $(X_s,f_s)$ is a perturbation given by a stabilisation of $(X,f)$ with $s\neq0$, the image of $f_s$ intersected with a Milnor ball has the homotopy type of a wedge of spheres of dimension $n$.
\end{proposition}

We used the same techniques to prove a similar result for the smooth case, in \cite[p.4]{GimenezConejero2021}. There are routine technical details of the proof that can be found there. In any case, to prove this, note that the case of $\eqA_e$-codimension equal to $0$ is trivial. Otherwise, $X_s$ is smooth and $f_s$ is stable outside the origin by \cite[Theorem 1.4]{Mond1994}, so we can take a Milnor sphere $S_\epsilon$ such that the image of the stabilisation of the pair $(X,f)$ is topologically trivial, seen as the zero-set of its defining equation $G$. One can conclude applying \cref{siersma} to prove that the stable perturbation has the homotopy type of a wedge of spheres.

\begin{definition}
For $f:(X,S)\rightarrow  (\CC^{n+1},0)$ as in \cref{muICIS}, if $(X_s,f_s)$ is a perturbation given from a stabilisation of $f$, we will say that the \textit{image Milnor number} of $(X,f)$ is the number of spheres, in the homotopy type, of the image of $(X_s,f_s)$ on a Milnor ball, for $s\neq 0$ (see \cref{fig:muICIS1}). This number will be denoted by  $\mu_I(X,f)$.
\end{definition}

\begin{remark}
Actually, an equivalent definition can be given replacing $(X_s,f_s)$ from the stabilisation for any stable $(X_u,f_u)$ from a versal unfolding (or, in general, a stable unfolding) because any stabilisation can be found inside a versal unfolding by means of a base change. This may simplify some arguments or intuitions and we will use both indistinctly.

For the same reason, the definition does not depend on the stabilisation (cf. \cite[p. 12]{Mond1994}). Finally, a stabilisation always exists when $f$ has corank one or $(n,n+1)$ are nice dimensions. In fact, \textit{the bifurcation set} $\mathcal B$ of a versal unfolding $(\mathcal{X},\pi,F,j)$ over $(\CC^d,0)$ is the set germ in $(\CC^d,0)$ of parameters $u$ such that $(X_u,f_u)$ has some instability. It is enough to show that $\mathcal B$ is analytic and proper in $(\CC^d,0)$.

On one hand, we consider the set germ $\mathcal C$ in $(\CC^p\times\CC^d,0)$ of pairs $(y,u)$ such that $(X_u,f_u)$ is unstable at $y$. We fix a small enough representative $F:\mathcal X\to Y\times U$, where $Y$ and $U$ are open neighbourhoods of the origin in $\CC^p$ and $\CC^d$, respectively.
Then $\mathcal C$ is the support of the relative normal module on $Y\times U$, defined as
$$ N\eqA_e (F|U) \coloneqq \frac{\theta(F|U)}{t_{rel}F(\theta_{\mathcal X|U})+\omega_{rel} F(\theta_{Y\times U|U})}, $$
where $\theta(F|U)$, $\theta_{\mathcal X|U}$ and $\theta_{Y\times U|U})$ are the submodules of $\theta(F)$, $\theta_{X}$ and $\theta_{Y\times U}$ of relative vector fields, respectively, and $t_{rel}(F)$ and $\omega_{rel}(F)$ are the respective restrictions of $tF$ and $\omega F$. The fact that $(X,f)$ has finite singularity type implies that $N\eqA_e (F|U)$ is coherent (see the proof of \cite[Lemma 5.3]{Mond-Nuno2020}) and, hence, $\mathcal C$ is analytic in $Y\times U$. Moreover, the projection $\pi_2:\mathcal C\to U$ given by $\pi_2(y,u)=u$ is a finite mapping, because $(X,f)$ has isolated instability. Therefore, $\mathcal B=\pi(\mathcal{C})$  is also analytic in $U$, by Remmert's finite mapping theorem.

On the other hand, we prove that $\mathcal B$ cannot be equal to $U$. Since $(\mathcal{X},\pi,F,j)$ is a versal unfolding of $(X,f)$, $(\mathcal X,\pi)$ is a versal unfolding of $X$. Hence, there exists $u_0\in U$ such that $X_{u_0}$ is smooth. Now, we can apply the classical Thom-Mather theory to the mapping $f_{u_0}\colon X_{u_0}\to\CC^p$. If either $(n,p)$ are nice dimensions or $f_{u_0}$ has only corank one singularities, then for almost any $u$ in a neighbourhood of $u_0$, the mapping $f_u\colon X_u\to\CC^p$ has only stable singularities (see, for example, \cite[Propositions 5.5 and 5.6]{Mond-Nuno2020}).


\end{remark}

\begin{figure}[H]
	\centering
		\includegraphics[width=0.8\textwidth]{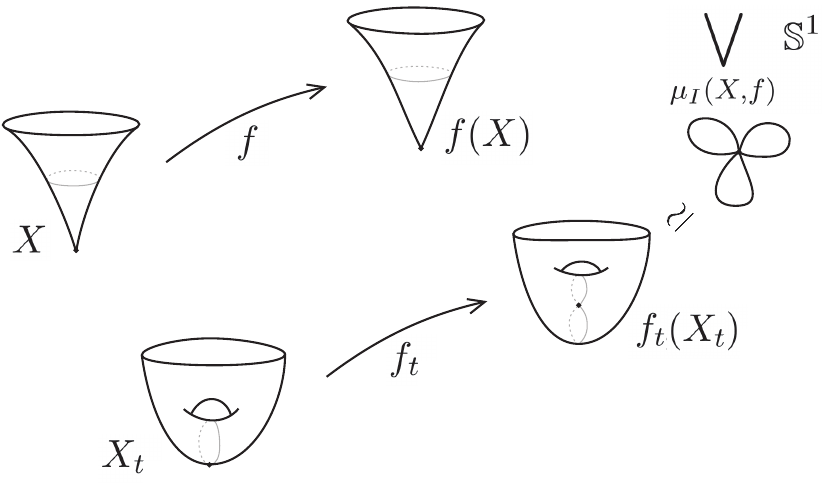}
		\vspace{2mm}
	\caption{Illustration of how $\mu_I(X,f)$ works, i.e., of the homology of the image of a stable perturbation $(X_t,f_t)$.}
	\label{fig:muICIS1}
\end{figure}
%

A desirable property of this topological $\eqA$-invariant is that it is conservative, as it was for the usual image Milnor number (see \cite[Theorem 2.6]{GimenezConejero2021}). The reasoning that proves the conservation of the usual image Milnor number can be applied verbatim for the general version, and is based as well on \cref{siersma}. Here, we give a sketch of the proof.

\begin{theorem}\label{conservation}
Let $f:(X,S)\rightarrow  (\CC^{n+1},0)$ be as in \cref{muICIS}, and $(X_{u_0},f_{u_0})$ a perturbation in a one-dimensional unfolding of $(X,f)$. Take a representative of the unfolding such that its codomain is a Milnor ball $B_\epsilon$.  Then
$$ \mu_I(X,f)=\beta_n\big(f_{u_0}(X_{u_0})\big)+\sum_{y\in B_\epsilon}\mu_I(X_{u_0},f_{u_0}; y),$$
where $\beta_n$ is the $n$th Betti number, if $u_0$ is small enough.
\label{conservationICIS}
\end{theorem}
	
\textit{Sketch of the proof.}
If $(X_{u_0},f_{u_0})$ is stable the result is trivial. 

Assume that $(X_{u_0},f_{u_0})$ is not stable. Then, take a versal unfolding of $(X,f)$ such that it unfolds the original one-dimensional unfolding and $(X_{u_{0},v},f_{u_0,v})$ is stable for $v\neq 0$ small enough. Consider the defining equations for $f_{u,v}(X_{u,v})$, $G$. Now, as $f$ is stable outside the origin and $X$ is smooth outside the points of $S$ we can take a Milnor radius $\epsilon$ such that the family of equations $G$ is topologically trivial over $S_\epsilon$. Now we are in the conditions of applying \cref{siersma} and follow the reasoning of \cite[Theorem 2.6]{GimenezConejero2021}, but working on $X_{u_{0},v}$ and the corresponding instabilities of $(X_{u_0},f_{u_0})$.

\qed\newline

In particular this implies the upper semi-continuity of the image Milnor number.

\begin{corollary}\label{uscontinuity}
Using the notation and hypotheses of \cref{conservation}, $\mu_I(X,f)$ is upper semi-continuous, i.e., $$\mu_I(X,f)\geq \mu_I(X_{u_0},f_{u_0};y).$$
\end{corollary}

\section{Multiple points and the \textsc{icss}}\label{sec:3}

One may ask what happens when we have a map germ $f:(X,S)\rightarrow  (\CC^{p},0)$ with $X$ \textsc{icis} but $\dim X =n<p$, in general. Houston studied this for the case of smooth source in \cite{Houston2010} using the multiple point spaces of the map germ and an Image-Computing Spectral Sequence (\textsc{icss}). We take a similar path, therefore, we need the machinery of the multiple point spaces (see \cite{Nuno2017, Mond1987, Marar1989}, among others, for more details). To simplify notation, $(X_t,f_t)$ will be a stable perturbation of $(X,f)$.

\begin{definition}\label{multiple}
The \emph{$k$th-multiple point space}, $D^k(f)$, of a mapping or a map germ $f$ is defined as follows:
\begin{enumerate}[\itshape(i)]
	
\item Let $f\colon X\rightarrow Y$ be a locally stable mapping between complex manifolds. Then, $D^k(f)$ is equal to the closure of the set of points $\left(x_1,\dots,x_k\right)$ in $X^k$ such that $f\left(x_i\right)=f\left(x_j\right)$ but $x_i\neq x_j$, for all $i\neq j$.

	\item Given $f:(X,S)\rightarrow(\CC^p,0)$ with finite singularity type, if $(\mathcal{X},\pi,F,j)$ is a stable unfolding then $$D^k(f)=J_k^{-1}\medpar{D^k(F)},$$ where $J_k=\overbrace{j\times\dots\times j}^{k}$.\label{itemdk} 
	
	\item For a map germ $f\colon(X,S)\rightarrow(\CC^p,0)$, we will denote as $d(f)$ the maximal multiplicity of the stable perturbation of $f$ (i.e., $d(f)\coloneqq\textnormal{max}\left\{k : D^k(f_t)\neq\varnothing\right\}$) and $s(f)$ the number of points of the set $S$.\label{itemdsf}
\end{enumerate}
\end{definition}

These definitions behave properly under isomorphisms and base change of unfoldings, i.e., \cref{itemdk,itemdsf} do not depend on the stable unfolding (this is proved in \cite[Lemma 2.3]{Nuno2017}, which is stated for the smooth case but the proof works for our case).

The multiple point spaces have some useful properties, for example we study their complex structure if $f$ is of corank one in \cref{DICIS}. We know some facts about them in any corank, for example, we can use \cite[Theorem 4.3 and Corollary 4.4]{Houston1997} with an $\eqA$-finite map germ $f:(X,S)\to (\CC^p,0)$, $\dim X=n$, and deduce that the dimension of $D^k(f)$ is $p-k(p-n)$ (if this number is not negative nor $D^k(f)$ is empty), because the same proofs work. Furthermore, in the pair of dimensions $(n,p)$, we deduce from \cite[Theorem 4.3]{Houston1997} that $d(f)$ is at most the integer part of $\frac{p}{p-n}$. Finally, taking into account these previous remarks, this maximum is attained when $s(f)\geq d(f)$ because the proof of \cite[Lemma 3.3]{GimenezConejero2021} only relies on the dimension of the multiple point spaces.


We want to study the multiple point spaces because they provide a lot of information about the map they come from by means of an \textsc{icss}, as one can see in \cite{Marar1989} or \cite{Goryunov1995}. The following lemma is an example of this, and it generalizes \cite[Theorem 2.14]{Marar1989} for the context of multi-germs and an \textsc{icis} in the source and \cite[Theorem 2.4 and Corollary 2.6]{Houston2010} for the context of \textsc{icis} in the source.

\begin{lemma}\label{DICIS}
For $f:(X,S)\rightarrow (\CC^p,0)$ of corank $1$ and finite singularity type, $\dim X =n < p$:
\begin{enumerate}[\itshape(i)]
	\item $(X,f)$ is stable if, and only if, $D^k(f)$ is smooth of dimension $p-k(p-n)$, or empty, for $k\geq1$.
	\item If $\eqA_e$-codim$(X,f)$ is finite, for each $k$ with $p-k(p-n)\geq 0$, $D^k(f)$ is empty or an \textsc{icis} of dimension $p-k(p-n)$. Furthermore, for those $k$ such that $p-k(p-n)< 0$, $D^k(f)$ is a a subset of $S^k$, possibly empty.
\end{enumerate}
\end{lemma}
\begin{proof}
For the first statement, if $(X,f)$ is stable, $X$ is smooth and it follows from \cite[Theorem 2.4]{Houston2010} (cf. \cite[Theorem 2.14]{Marar1989} for the mono-germ case). For the converse, if every $D^k(f)$ is smooth, then, in particular, so is $D^1(f)=X$ and, again, the result follows from \cite[Theorem 2.4]{Houston2010}.

For the second statement, note that for the case of $X$ being smooth the statement is contained in \cite[Corollary 2.6]{Houston2010}. Fortunately, if we take a versal unfolding $(\mathcal{X},\pi,F,j)$ of $(X,f)$ we can apply that result. Also, for $k$ with $p-k(p-n)\geq 0$, observe that the codimension of $D^k(F)$ coincides with the one of $D^k(f)$ because outside the isolated singularities they are smooth and of the same codimension. Furthermore, the only singularities that can appear in $D^k(f)$ are at the problematic points (the ones that come from $S$) by the Mather-Gaffney criterion \cref{MGcriterion}. It only remains to check that $D^k(f)$ is a complete intersection, and this is the case as it can be constructed as a pull back of a complete intersection, $D^k(F)$, and both have the same codimension.

The other case is trivial. 
\qed\newline
\end{proof}

Note that this is the best we can aim for: we need to study $D^1(f)=X$. The following example illustrates this.

\begin{example}\label{inmersion}
Let $(X,0)\subset(\CC^N,0)$ be a germ of an \textsc{icis} of dimension $n$. Then, the inclusion $i:(X,0)\rightarrow(\CC^N,0)$ is stable in the sense that the $\eqA_e$\mdash codimension of $i$ is zero, since $\omega i$ is surjective. However, the $\eqA_e$\mdash codimension of $(X,i)$ is equal to the Tjurina number of $(X,0)$. Furthermore, after taking a stable perturbation of $(X,i)$, say $(X_t,i_t)$, every $D^k(i_t)$ is empty for $k\geq 2$  because $X_t$ is smooth for $t\neq0$ and $i_t$ is an inclusion to $\CC^N$.
\end{example}

The multiple point spaces in corank one are specially friendly, as we have seen. Indeed, if we consider $(X,S)\subseteq (\CC^N,S)$ and a map germ $f:(X,S)\rightarrow (\CC^p,0)$ of corank $1$ and finite singularity type, $\dim X =n < p$, we can simplify even more the structure of the different $D^k(f)$. This is because $D^k(f)$ is a subset of $X^k$, therefore a subset of $\CC^{Nk}$ and, if $f$ is of corank one, we can assume that $f$ has the form
\begin{align*}
	f:(X,S)&\to(\CC^p,0)\\
	(x_1,\dots,x_N)&\mapsto\medpar{x_1,\dots,x_{N-1},h_1(x),h_2(x)}.
\end{align*}
Therefore,  the first $N-1$ coordinates of all the points in a $k$-tuple in $D^k(f)$ must be the same. Omitting repeated entries, we can see $D^k(f)$ as a subset of $X\times\CC^{k-1}$. Finally, observe that this identification preserves the \textsc{icis} structure.

On the other hand, there is a natural action of $\Sigma_k$ in $D^k(f)$ by permutation of entries of the $k$-tuples of points. Also, recall that all the elements in $\Sigma_k$ can be decomposed into disjoint cycles in a unique way, called the \textit{cycle decomposition}, and this inspires a refinement of the $k$th-multiple point space based on the relations an element $\sigma\in\Sigma_k$ gives. To be precise, if we take a partition of $k$, $\gamma(k)=(r_1,\dots, r_m)$, and $\alpha_i=\texttt{\#}\left\{j:r_j=i\right\}$, where $\texttt{\#}A$ denotes the number of points in $A$, one can find an element $\sigma\in\Sigma_k$ such that it can be decomposed into $\alpha_i$ cycles of length $i$, all the cycles being pair-wise disjoint. In this case, the partition $\gamma(k)$ is called the \textit{cycle type} of $\sigma$ (see \cite[pp. 2--3]{Sagan2001}).

\begin{definition}
We define $D^k\medpar{f,\gamma(k)}$ as (the isomorphism type of) the subspace of $D^k(f)$ given by the fixed points of $\sigma$ with the usual action, for $\sigma$ of cycle type $\gamma(k)$. We may also use $D^k(f)^\sigma$ instead of $D^k\medpar{f,\gamma(k)}$ to specify the element.
\end{definition}

\begin{remark}
By symmetry, $D^k(f)^\sigma$ is isomorphic to $D^k(f)^{\sigma'}$ if $\sigma$ and $\sigma'$ have the same cycle type (recall that the cycle types determine the conjugacy classes). Hence the definition of $D^k\medpar{f,\gamma(k)}$ works modulo isomorphism. We will omit this detail in general (see also \cref{DP}).
\end{remark}

If we take into account the group action and the subspaces $D^k(f)^\sigma$, we have a refinement of \cref{DICIS}. It is a generalization of  \cite[Corollary 2.15]{Marar1989} for the context of multi-germs and an \textsc{icis} in the source and \cite[Corollary 2.8]{Houston2010} for the context of \textsc{icis} in the source.

\begin{lemma}\label{lem: dkgamma}
With the hypotheses of \cref{DICIS} and $\gamma(k)$ a partition of $k$, we have the following.
\begin{enumerate}[\itshape(i)]
	\item If $f$ is stable, $D^k\medpar{f,\gamma(k)}$ is smooth of dimension $p-k(p-n)-k+\sum_i\alpha_i$, or empty.\label{iPD}
	\item If $\eqA_e-codim(X,f)$ is finite, then:\label{iiPD}
	\begin{enumerate}[(a)]
		\item for each $k$ with $p-k(p-n)-k+\sum_i\alpha_i\geq0$, $D^k\medpar{f,\gamma(k)}$ is empty or an \textsc{icis} of dimension $p-k(p-n)-k+\sum_i\alpha_i$,
		\item for each $k$ with $p-k(p-n)-k+\sum_i\alpha_i<0$, $D^k\medpar{f,\gamma(k)}$ is subset of $S^k$, possibly empty.
	\end{enumerate}
\end{enumerate}
\label{DP}
\end{lemma}

A proof of this lemma can be seen in \cite[Corollary 2.8]{Houston2010}, because the same proof applies once we know \cref{DICIS}. In any case, observe that the \cref{iPD} is a consequence of \cref{DICIS} and the fact that we are adding $k-\sum_i\alpha_i$ equations to the ones of $D^k(f)$ to form $D^k\medpar{f,\gamma(k)}$, as each $r_i$ of $\gamma(k)$ gives $r_i-1$ more equations. \cref{iiPD} follows the same idea of \cref{iPD}.

From the homology of the Milnor fiber of the multiple point spaces, a special part will serve our purposes, the so-called \textit{alternating part}.
\begin{definition}
Given a \textit{sign} homomorphism, $sign:G\rightarrow\left\{\pm1\right\}$ where $\left\{\pm1\right\}\cong\ZZ/2\ZZ$, and a linear action of a finite group $G$ on some $\CC$\mdash vectorial space $H$, we say that the $G$\textit{-alternating part} of $H$ is the set
$$ \left\{h\in H : gh=sign(g)h, \text{ for all }g\in G\right\},$$
and we denote it by $H^{\Alt_G}$. If the group is $\Sigma_k$, then the $sign$ homomorphism is the usual $sign$ of a permutation and we simply write $H^{\Alt_k}$ or $H^\Alt$ if the group is clear from the context.
\end{definition}

\begin{remark}
In terms of representation theory, $H^{\Alt_G}$ is the \textit{isotype} of the sign representation of the representation $H$. Furthermore, the $sign$ homomorphism can be defined as the usual $sign$ for permutations for every finite group, seen as a subgroup of a $\Sigma_N$ by Cayley's theorem (see \cite[Proposition 1.6.8]{Robinson1996}).
\end{remark}

With the study of the multiple point spaces, we are able to obtain a lot of information of images of stable perturbations in any pair of dimensions, as long as $n<p$. This is done by means of an \textsc{icss}, as the one we show below. For example, in \cite{Houston2010}, Houston uses a similar theorem as we are going to use this theorem now, furthermore it will appear in later techniques as well. The theorems that give an \textsc{icss} have been evolving through the years, being more general and with new approaches. A first version of these theorems was \cite[Proposition 2.3]{Goryunov1993} for rational homology and $\eqA$-finite map germs. See also the generalizations \cite[Corollary 1.2.2]{Goryunov1995} for finite maps and integer homology; \cite[Theorem 5.4]{Houston2007} with a bigger class of maps, any coefficients and homology of the pair; \cite{CisnerosMolina2019} with a new approach and \cite[Section 10]{Mond-Nuno2020} for a self-contained review of the \textsc{icss}.

\begin{theorem}\label{SecuenciaH}
Let $F:X\rightarrow Y$ be a finite and proper subanalytic map and let $Z$ be a, possibly empty, subanalytic subset of $X$ such that $\left.F\right|_Z$ is also proper. Then, there exists a spectral sequence
$$ E^1_{r,q}= H^{\Alt_{r+1}}_q\medpar{D^{r+1}\left(F\right),D^{r+1}\left(\left.F\right|_Z\right);\ZZ}\Longrightarrow H_*\medpar{F\left(X\right),F\left(Z\right);\ZZ},$$
and the differential is induced from the natural map $\epsilon_{r+1,r}:D^{r+1}(F)\rightarrow D^{r}(F)$ given by the projection on the first $r$ coordinates.
\end{theorem}

\begin{remark}
We skip the details about the difference between the homology of the \textit{alternating chains}, usually called the \textit{alternating homology}, and the \textit{alternating part of the homology} we use here. There is no relevant change using one or the other here except, perhaps, in \cref{corrangomayor}, where the details of the difference are explained. The reason we use them indistinctly is that they coincide with rational coefficients, for more information see \cite[Theorem 2.1.2]{Goryunov1995} for our setting, but also \cite[Proposition 10.1]{Mond-Nuno2020}. A complete account of this can also be found in the thesis of the first author (in particular, \cite[Section 2.2]{Robertothesis}).
\end{remark}

Our first application of \cref{SecuenciaH} is the following result, which follows the idea of \cite[Theorem 3.1]{Houston2010} and generalizes it  when the source is an \textsc{icis} and we consider integer homology. Moreover, it is a generalization of \cite[Theorem 3.3.1]{Goryunov1995}, for multi-germs. 

\begin{theorem}\label{homotopydis}
Consider a map germ $f:(X,S)\rightarrow  (\CC^{p},0)$ of finite $\eqA_e$-codimension and of corank $1$, with $X$ \textsc{icis} of dimension $\dim X=n<p$. Then, the reduced integer homology of the image of a stable perturbation of $(X,f)$ is zero except possibly in dimensions 
\begin{enumerate}[\itshape(i)]
	\item $p-k(p-n)+k-1$ for all $2\leq k \leq d(f)$,\\
	\item $d(f)-1$  if $s(f)>d(f)$, and\\
  \item $n$ if $X$ is non-smooth (or $p = n + 1$).
\end{enumerate}
\end{theorem}
\begin{proof}
Apply \cref{SecuenciaH} to a versal unfolding of $(X,f)$, say $(\mathcal{X},\pi,F,j)$, and its restriction to $Z=\pi^{-1}(t)=X_t$ that gives a stable perturbation, $(X_t,f_t)$. Hence, we have the spectral sequence
$$ E^1_{r,q}\left(F,f\right)\coloneqq H^{\Alt_{r+1}}_q\medpar{D^{r+1}\left(F\right),D^{r+1}\left(f_t\right);\ZZ}\Longrightarrow H_*\medpar{F\left(\mathcal{X}\right),f_t\left(X_t\right);\ZZ}. $$

This spectral sequence, of homology type, collapses at the second page instead of the first one (cf. \cref{1pag}), because the bottom row of the $E_1$ page may have non-trivial differentials, as in the case of the quadruple point singularity of a map $\CC^2\rightarrow\CC^3$. From this, we can recover the limit of the spectral sequence and deduce the result. First of all, take into account that $\textnormal{Im}(F)$ is contractible. Furthermore, by \cref{DICIS}, the reduced homology of $D^k(f_t)$ could be non-trivial only in middle dimension, therefore, the groups
$$H_i\medpar{F\left(\mathcal{X}\right),f_t\left(X_t\right);\ZZ}$$
are possibly non-trivial when
\begin{align*}
	i&=r+ \dim D^{r+1}+1\\
	&= r+ p-(r+1)(p-n)+1\\
	&= p-(p-n)(r+1)+(r+1),
\end{align*}
for $2\leq r+1\leq d(f)$. 

The only other possible non-trivial entry, after collapsing the sequence, is $E^2_{d(f),0}$. This comes from the fact that the bottom row of the first page is an exact sequence except at $E^1_{d(f),0}$. This, in turn, comes from applying \cref{SecuenciaH} for $F$ and $Z=\varnothing$, deducing that the non-trivial part of the bottom row has to be exact because $F(\mathcal{X})$ is contractible (one can also apply the proof and statement of \cite[Lemma 3.3]{Houston2010} verbatim for this case).

Finally, when $r=0$ we have some homology apart from the 0-dimensional, because $D^1(f_t)=X_t$ and it is the stable perturbation of the \textsc{icis} $X$. The homology in this case appears when $i=0+\dim(X)+1$ and it is equal to $\mu(X)$.
Using the exact sequence of the homology of the pair, the result follows.
\qed\newline
\end{proof}

The argument of \cref{homotopydis} would work for any corank if we were able to prove that the alternating homology of the pairs $\medpar{D^k(F),D^k(f)}$ in any corank appear in the same dimensions as in corank one. Unfortunately, the techniques used for the smooth case do not give what we hope for.

\begin{table}[htb]
\begin{minipage}{0.41\textwidth}
\centering
\scalebox{0.81}{
\begin{tabular}{c | c@{\hspace{1.7\tabcolsep}}  c@{\hspace{1.7\tabcolsep}}  c@{\hspace{1.7\tabcolsep}}  c@{\hspace{1.7\tabcolsep}}  c@{\hspace{1.7\tabcolsep}}  c@{\hspace{1.7\tabcolsep}}  c@{\hspace{1.7\tabcolsep}}  c@{\hspace{1.7\tabcolsep}}  c@{\hspace{1.7\tabcolsep}}  c}
		$7	 $&$       $&$        $&$ 			$&$ 			$&$			 $&  			&$				$& &$				$&\\[1.5pt]
		$6	 $&$    \bullet   $&$        $&$ 			$&$ 			$&$			 $&  			&$				$&&$				$&\\[1.5pt]
		$5 $&$       $&$     \bullet   $&$ 			$&$ 			$&$			 $&  			&$				$&&$				$&\\[1.5pt]
		$4	 $&$       $&$        $&$ 		\bullet	$&$ 			$&$			 $&  			&$				$&&$				$&\\[1.5pt]
		$3	 $&$       $&$        $&$ 			$&$ 	\bullet		$&$			 $&  			&$				$&&$				$&\\[1.5pt]
		$2	 $&$       $&$        $&$ 			$&$ 			$&$		\bullet	 $&  			&$				$&&$				$&\\[1.5pt]
		$1	 $&$     $&$        $&$ 			$&$ 			$&$			 $&$ 	\bullet	$	&$				$&&$				$&\\[1.5pt] 
		$0	 $&$     $&$        $&$ 			$&$ 			$&$			 $& 			&$		\bullet		$&$\bullet$&$			\bullet	$&$\cdots$\\ \hline
			  \diagbox[dir=SW,innerwidth=0.4cm]{$r$}{$q$}      &$0$  &$1$&$2$&$3$&$4$&$5$&$6$&$7$&$8$&$\cdots$
		\end{tabular}}
		\end{minipage}%
		\begin{minipage}{0.59\textwidth}
		\centering
\vspace{-3.2mm}
		\scalebox{0.81}{\begin{tabular}{c | c@{\hspace{1\tabcolsep}}  c@{\hspace{1\tabcolsep}}  c@{\hspace{1\tabcolsep}}  c@{\hspace{1\tabcolsep}}  c@{\hspace{1\tabcolsep}}  c@{\hspace{1\tabcolsep}}  c@{\hspace{1\tabcolsep}}  c@{\hspace{1\tabcolsep}}  c@{\hspace{1.1\tabcolsep}} c}
		$H^{\Alt}_7	 $& & & & & &  			&$				$&$				$&$				$&$				$\\[1.5pt]
		$H^{\Alt}_6	 $&$       $&$        $&$ 			$&$ 			$&$			 $&  			&$				$&$				$&$				$&$				$\\[1.5pt]
		$H^{\Alt}_5	 $&$   X_t   $&$        $&$ 			$&$ 			$&$			 $&  			&$				$&$				$&$				$&$				$\\[1.5pt]
		$H^{\Alt}_4 $&$       $&$    f_t   $&$ 			$&$ 			$&$			 $&  			&$				$&$				$&$				$&$				$\\[1.5pt]
		$H^{\Alt}_3	 $&$       $&$        $&$ 		f_t	$&$ 			$&$			 $&  			&$				$&$				$&$				$&$				$\\[1.5pt]
		$H^{\Alt}_2	 $&$       $&$        $&$ 			$&$ 	f_t		$&$			 $&  			&$				$&$				$&$				$&$				$\\[1.5pt]
		$H^{\Alt}_1	 $&$       $&$        $&$ 			$&$ 			$&$		f_t $&$ 	\left(F,f_t\right) 	$&$				$&$				$&$				$&$				$\\[1pt]
		$H^{\Alt}_0	 $&$     $&$        $&$ 			$&$ 			$&$			 $&  			&$		F		$&$		F		$&$		F		$&$		\cdots		$ \\[0.15mm] \hline  \rule{0pt}{2.6ex}
			                     &$D^1$&$D^2$&$D^3$&$D^4$&$D^5$&$D^6$&$D^7$&$		D^8		$&$		D^9		$&$		\cdots		$\\
		\end{tabular}
}\end{minipage}
\caption{First page of the spectral sequence $E^1_{r,q}= H^{\Alt_{r+1}}_q\left(D^{r+1}\left(F\right),D^{r+1}\left(f_t\right)\right)$ for a map germ $f:(X,S)\rightarrow(\CC^6,0)$ in the pair of dimensions $(5,6)$ (left) and the \textsl{input} objects required to compute it after the identifications $H_m^\Alt\big(D^{r+1}(F),D^{r+1}(f_t)\big)\cong H_{m-1}^\Alt\big(D^{r+1}(f_t)\big)$, $m>1$, and $H_\bullet^\Alt\big(D^{r+1}(F),D^{r+1}(f_t)\big)= H_\bullet^\Alt\big(D^{r+1}(F)\big)$, $r+1>d(f)$ (right).}
\label{1pag}
\end{table}


The following lemma is stated for rational homology, the version with integer coefficients is explained in the remark that follows its proof.

\begin{lemma}[cf. {\cite[Theorem 4.6]{Houston1997}}] \label{corrangomayor}
Let $f:(X,S)\to(\CC^p,0)$ be an $\eqA$-finite map germ, where $(X,S)$ is a germ of an \textsc{icis} of dimension $n<p$ and codimension $r$. Consider a non-empty $D^k(f_t)$ and write $d$ for $\dim_\CC D^k(f_t)$. Then,
\begin{enumerate}[\itshape(i)]
	\item for $k\geq2$, $H^\Alt_q\medpar{D^k(f_t)}=0$ if $q\neq 0$ or $q\notin \left[d+(1-r)k,d\right]$, and
	\item $H^\Alt_q\medpar{D^1(f_t)}$ is zero for $q\neq0,d$.\label{item: alternatinghomology}
\end{enumerate}
\end{lemma}
\begin{proof}
Observe that $D^1(f_t)=X_t$ is the Milnor fibre of the \textsc{icis} $D^1(f)=X$, so we can assume that $k\geq2$. 

With rational coefficients, the homology of the alternating chains and the alternating part of the homology coincide (see \cite[Proposition 10.1]{Mond-Nuno2020}), so we can use them indistinctly. In this proof we use the homology of the alternating chains to prove the statement.

For $q\geq d$ the space $D^k(f_t)$ has the homotopy type of a $CW$-complex of dimension $d$, therefore, there are no alternating chains above $\dim_\CC D^k(f_t)$. This proves the result for those $q$. It only remains to check when $q<\dim_\CC D^k(f_t)$ for $k\geq2$. To prove the remaining part of the lemma, we take the proof of \cite[Theorem 4.6]{Houston1997} as reference. This argument consists of two steps: controlling the alternating homology of the pair $\medpar{D^k(F),D^k(f_t)}$ by virtue of \cite[Theorem 3.30]{Houston1997}, where $F$ is a one parameter unfolding of $f$, and specify exactly when this pair can have alternating homology using \cite[Theorem 3.13]{Houston1997}. 

The hypothesis of \cite[Theorem 3.13 and Theorem 3.30]{Houston1997} are not too restrictive, so we can take a one-parameter unfolding $F$ of $f$ and combine these theorems to prove that
$$H^\Alt_q\medpar{D^k(F),D^k(f_t)}=0 $$
for
\begin{align*} q&\leq \min\left\{(n+1-r+1)k-(p+1)(k-1)-1, nk-p(k-1)\right\} \\
&= nk-p(k-1)+(1-r)k=d+(1-r)k.
\end{align*}
Therefore, using the exact sequence of the pair and the fact that $D^k(F)$ contracts to isolated points in an equivariant way, we have that $H^\Alt_0\medpar{D^k(f_t)}\cong H^\Alt_0\medpar{D^k(F)}$ and $H^\Alt_{q-1}\medpar{D^k(f_t)}\cong H^\Alt_{q}\medpar{D^k(F),D^k(f_t)}$, for $q>1$.
\qed\newline
\end{proof}

\begin{remark}
As the reader should notice, this lemma can be stated for the homology of alternating chains (some times denoted by $AH$) instead of alternating homology and integer homology instead of rational homology. If one wants to use integer homology the same proof works changing $H^\Alt$ for $AH$ and adding in \cref{item: alternatinghomology} that the homology is free if $q=0,d$ (which is trivial, because the group acting is $\Sigma_1$). Nevertheless, \cref{homotopydiscorankr} is well stated as it is, changing to rational coefficients to write $H^\Alt$.
\end{remark}

\begin{corollary}\label{homotopydiscorankr}
With the hypotheses of \cref{corrangomayor}, the reduced integer homology of the image of a stable perturbation of $(X,f)$ is zero except possibly in dimensions 
\begin{enumerate}[\itshape(i)]
	\item $p-k(p-n)+k-1+s$ for all $0\leq s\leq (1-r)k$ and $2\leq k \leq d(f)$,\\
	\item $d(f)-1$  if $s(f)>d(f)$, and\\
  \item $n$ if $X$ is non-smooth (or $p = n + 1$).
\end{enumerate}
\end{corollary}
\begin{proof}
The proof follows from the version of \cref{corrangomayor} with integer coefficients and the homology of the alternating chains, \cref{SecuenciaH} and a careful inspection of the \textsc{icss} as in \cref{homotopydis}.
\qed\newline
\end{proof}

\begin{remark}
Observe that, if $X$ is a hypersurface, this theorem proves that the homology of the image appears in the same dimensions than the smooth case. Also, note that this theorem may not be sharp, because \cite[Theorem 3.13]{Houston1997} only gives a estimate to control the alternating homology of the pair $\medpar{D^k(F),D^k(f_t)}$ and it could be a bad estimate in general.
\end{remark}

 It is surprising that, in the case of corank one, the same proof does not prove \cref{homotopydis}. This makes us think that there is an argument that avoids the detail of the codimension of the \textsc{icis}:

\begin{conjecture}
Consider a map germ $f:(X,S)\rightarrow  (\CC^{p},0)$ of finite $\eqA_e$-codimension, with $X$ \textsc{icis} of dimension $\dim X=n<p$. Then, the reduced integer homology of the image of a stable perturbation of $(X,f)$ is zero except possibly in dimensions 
\begin{enumerate}[\itshape(i)]
	\item $p-k(p-n)k+k-1$ for all $2\leq k \leq d(f)$,\\
	\item $d(f)-1$  if $s(f)>d(f)$, and\\
  \item $n$ if $X$ is non-smooth.
\end{enumerate}
\end{conjecture}

\begin{remark}
This conjecture and \cref{{homotopydiscorankr,homotopydis}} are related to \cite[Theorems 2.3 and 2.8]{Liu2021}, when the source is smooth and the map germ is not necessarily $\eqA$-finite but the dimensions of the multiple point spaces are controlled. They, and \cref{Df}, are also closely related with \cite[Theorem 2.4]{Liu2021} in the particular case that $X$ is the double point space $D^2(f)$.
\end{remark}

Houston also uses \cref{SecuenciaH} in \cite{Houston2010} with a versal unfolding $F$, of a multi-germ $f:(\CC^n,S)\to(\CC^{n+1},0)$, and a section that gives the stable perturbation, $f_t$. Also, taking into account that the Euler-Poincaré characteristic of every page remains invariant, see \cite[Example 1.F]{McCleary2001}, and that the image of the versal unfolding is contractible, it remains to compute $\chi\left(E^1_{*,*}\right)$ to get $\mu_I(f)$, and the terms of the sum are arranged to define Houston's alternating Milnor numbers, $\mu_k^\Alt(f)$ (actually, he computes them through the limit of the spectral sequence, both ways give the same result). 

These ideas and the previous proofs inspire the following definition (see also the simplification and developments of it in \cite{GimenezConejero2021}, particularly how to determine $d(f)$ when $s(f)>d(f)$ with \cite[Lemma 3.3]{GimenezConejero2021}).

\begin{definition}\label{mukalticis}
Given an $\eqA$-finite map germ $f:(X,S)\rightarrow (\CC^{n+1},0)$ of corank one, with $X$ \textsc{icis} of dimension $n$, the $k$\textit{-th alternating Milnor number of }$(X,f)$, denoted as $\mu_k^\Alt(X,f)$, is defined by
	\[ \mu_k^\Alt(X,f)\coloneqq\begin{cases}
	\textnormal{rank } H^{\Alt_{k}}_{n-k+2}\medpar{D^k(F),D^k(f_t);\ZZ}, \quad  \quad\text{if $1\leq k\leq d(f)$}\\ \\
	\displaystyle{s(f)-1\choose d(f)}, \quad\quad\text{ if $k=d(f)+1$ and $s(f)>d(f)$}\\ \\
	0,\quad  \quad\text{otherwise},
	\end{cases}
\]
being $F$ its versal unfolding and $f_t$ a stable perturbation.
\end{definition}

These numbers are very useful because they decompose the image Milnor number. This was used by Houston in \cite[Definition 3.11]{Houston2010} for the smooth case of corank one, we extend it to the non-smooth case of corank one.

\begin{proposition}\label{Houston formula}
For $f:(X,S)\rightarrow (\CC^{n+1},0)$ $\eqA$-finite of corank one and $X$ an \textsc{icis} of dimension $n$,
$$ \sum_k \mu_k^\Alt\left(X,f\right)=\mu_I\left(X,f\right).$$
\end{proposition}
\begin{proof}
From the proof of \cref{homotopydis}, we only have to check that $\mu^\Alt_{d(f)+1}(X,f)$ coincides with the (rank of the) remaining non-zero entries of the spectral sequence after collapsing, i.e., we have to check that
$$\rank E^2_{d(f),0}={s\left(f\right)-1\choose d\left(f\right)}.$$
From \cite[Lemma 3.3]{Houston2010}, which can be stated for general stable map germs with a verbatim proof, or the constancy of the Euler-Poincaré characteristic of the spectral sequence (see \cite[Example 1.F]{McCleary2001}) we have
\[\pushQED{\qed} \rank E^2_{d(f),0}=\left|\sum_{\ell=d\left(f\right)+1}^{s\left(f\right)}\left(-1\right)^\ell{s\left(f\right) \choose \ell}\right|={s\left(f\right)-1\choose d\left(f\right)}.
\qedhere
\popQED\]
\end{proof}

One term deserves a bit of attention: $\mu_1(X,f)$. This term does not appear in the smooth case because it is zero but, in  general, it is equal to $\mu(X)$ (see \cref{fig:muICIS2}). \cref{Houston formula} allows us to reduce the weak form of Mond's conjecture for \textsc{icis} to the smooth case.

\begin{figure}[H]
	\centering
		\includegraphics[width=1\textwidth]{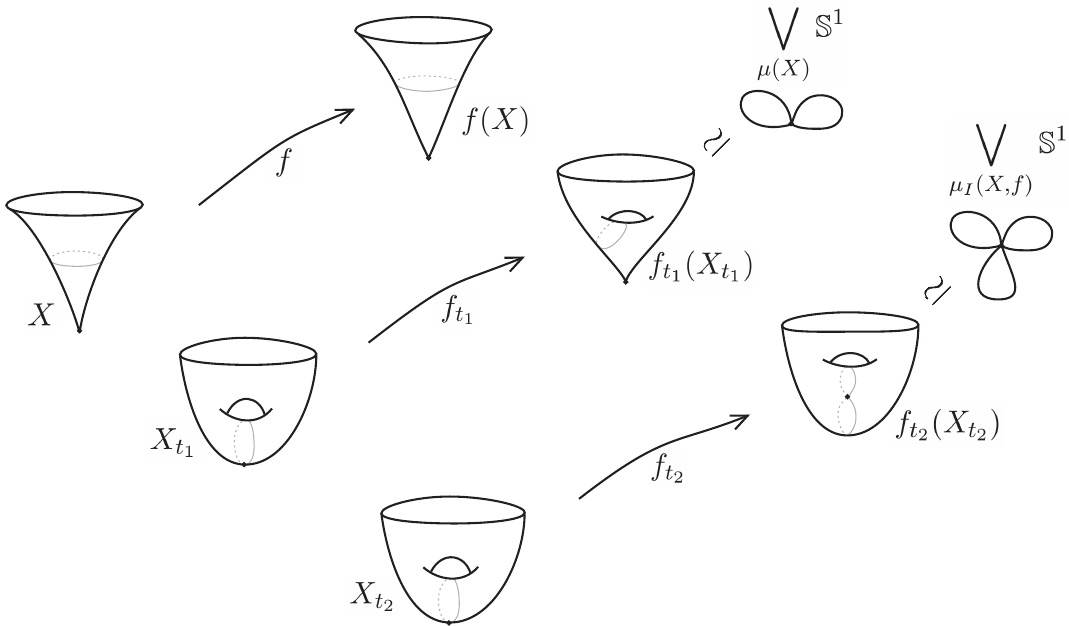}
	\caption{The first alternating Milnor number and its relationship with the deformations of $(X,f)$ and the $\mu_I(X,f)$. Here, one can also appreciate the conclusions of \cref{conservation,uscontinuity}.}
	\label{fig:muICIS2}
\end{figure}

\begin{corollary}[cf. {\cite[Theorem 3.9]{GimenezConejero2021}}] \label{CDMicis}
For $f:(X,S)\rightarrow (\CC^{n+1},0)$ $\eqA$-finite of corank one and $X$ an \textsc{icis} of dimension $n$, $\mu_I(X,f)=0$ if, and only if, $(X,f)$ is stable.
\end{corollary}
\begin{proof}
One direction is trivial.

If $\mu_I(X,f)=0$ then $\mu_1^\Alt(X,f)=\mu(X)=0$ and we are in the case of smooth domain, proved in \cite[Theorem 3.9]{GimenezConejero2021}.
\qed\newline
\end{proof}


\begin{remark}
Note that the weak form of Mond's conjecture for the smooth case in any corank implies, with the same proof of \cref{CDMicis}, the same conjecture for \textsc{icis} in any corank by means of \cref{Houston formula}, which can be stated for any corank using \cref{corrangomayor} and always carries a term equal to $\mu(X)$.
\end{remark}
\section{A Lê-Greuel type formula}\label{sec:Le-Greuel}

Now that we have a basic building of the image Milnor number with \textsc{icis} in the source, our last preparatory step is to prove a Lê-Greuel type formula for $\mu_I(X,f)$. In \cite{Nuno-Ballesteros2019} I. Pallarés-Torres and the second named author proved a Lê-Greuel type formula in the setting of the image Milnor number in the smooth case and finitely determined map germs. Recall the original Lê-Greuel formula, see \cite{Greuel1975,Trang1974a},
$$ \mu(X,0)+\mu(X\cap H,0)=\dim_\CC \frac{\mathcal{O}_n}{(g)+J(g,p)}.$$
Here, $(X,0)$ is an \textsc{icis} with defining equation $g$, $p$ a function such that $(X\cap H,0) $ is an \textsc{icis} as well, where  $H\coloneqq p^{-1}(0)$, and $J(g,p)$ is the ideal generated by the minors of maximum order of the Jacobian matrix of $(g,p)$. Taking into account that the right hand side of the equation could be seen as the number of critical points of $p$ restricted to the Milnor fiber of $X$, it is obvious that the main theorem of \cite{Nuno-Ballesteros2019} is a similar result for the context of map germs when $p$ is a generic linear projection (see \cref{figuralegreuel}):

\begin{theorem}[see {\cite[Theorem 3.2]{Nuno-Ballesteros2019}}] \label{LeGreuel}
Let $f:\Gzero{n}{n+1}$ be a corank 1 and $\eqA$-finite map germ with $n>1$. Let $p:\CC^{n+1}\rightarrow\CC$ be a generic linear projection which defines a transverse slice $g:\Gzero{n-1}{n}$. Then
$$\textnormal{\texttt{\#}}\Sigma\left(\left.p\right|_{Z_s}\right)=\mu_I\left(f\right)+\mu_I\left(g\right),$$
where $\textnormal{\texttt{\#}}\Sigma\left(\left.p\right|_{Z_s}\right)$ is the number of critical points on all the strata of $Z_s\coloneqq \text{Im}\left(f_s\right)$, being $f_s$ a stable perturbation of $f$.
\end{theorem}

If $p\colon\CC^{n+1}\rightarrow\CC$ is a generic linear projection, then the \emph{transverse slice} is, by definition, the restriction $f\colon\big(f^{-1}(H),0\big)\to(H,0)$, where $H$ is the hyperplane $H=p^{-1}(0)$. When $f$ has corank one, we can choose coordinates in such a way that the transverse slice is in fact a mapping $g:\Gzero{n-1}{n}$ and $f$ is an unfolding of $g$ (see \cite[pp. 1380--1381]{Marar2014}).

\begin{figure}[ht]
	\centering
		\includegraphics[width=0.85\textwidth]{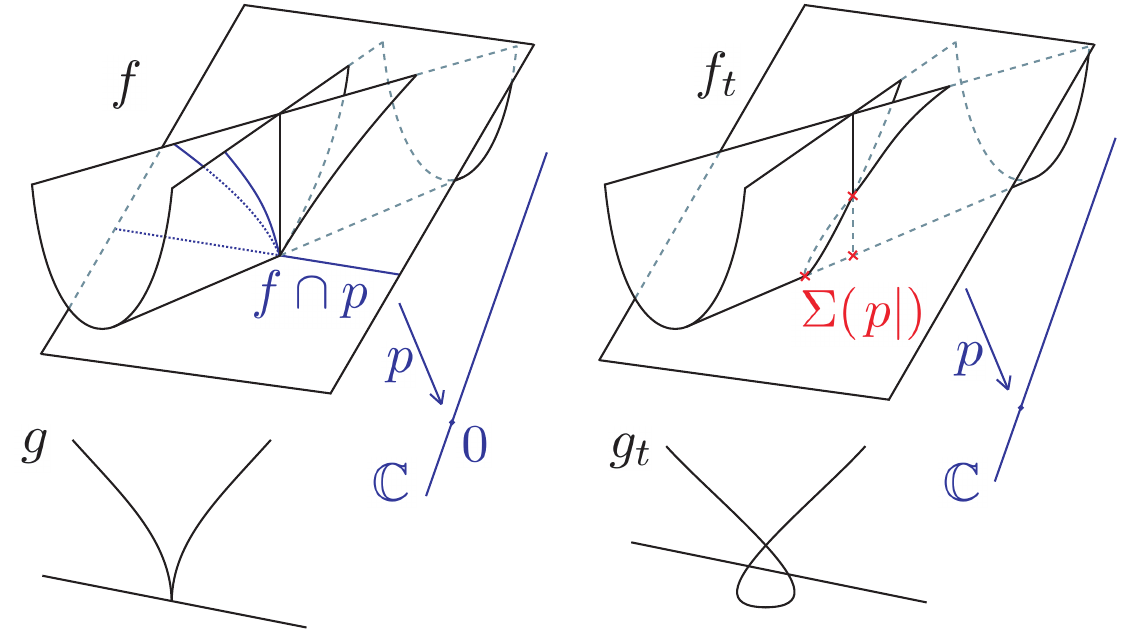}
		\caption{Depiction of the Lê-Greuel type formula for map germs.}
	\label{figuralegreuel}
\end{figure}

The stratification considered in the image of the stable perturbation $f_s$ in the theorem above is the \emph{stratification by stable types} (see \cite[Definition 7.2]{Mond-Nuno2020}). We recall that two points $y_1$ and $y_2$ in the image of $f_s$ belong to the same stratum if, and only if, the multi-germs of $f_s$ at such points are $\eqA$-equivalent. Since $f_s$ has only corank one singularities, each stratum is determined by a partition $\gamma(k)$, with $1\le k\le n+1$ and it is, in fact, a Whitney stratification (see \cite[Corollary 7.5]{Mond-Nuno2020}).

We prove a similar result for an \textsc{icis} in the source and multi-germs. The first step is to work with Marar's formula  for our setting of multi-germs and \textsc{icis} in the source. Fortunately, his proof is essentially combinatorial and one can prove the version we need with almost no modifications. Hence, if $\gamma(k)=(r_1,\dots,r_m)$ is a partition of $k$ and $\alpha_i=\texttt{\#}\left\{j:r_j=i\right\}$, Marar's formula is the following.

\begin{theorem}[cf. {\cite[Theorem 3.1]{Marar1991}}]\label{MararFormula}
Let $f_0:(\CC^n,0)\rightarrow (\CC^p,0)$ of corank 1 and $\eqA$-finite, $2\leq n<p$, and consider its stable perturbation $f_t:X_t\rightarrow \CC^p$. Then
$$ \chi\medpar{f_t(X_t)}=a_0\chi\left(U_t\right)+\sum_{k\geq 2} \sum_{\gamma(k)}a_{\gamma(k)}\chi\left(D^k\medpar{f_t,\gamma(k)}\right),$$
where $a_0=1$ and
$$ a_{\gamma(k)}=\frac{(-1)^{\sum \alpha_i+1}}{\prod_{i\geq1}i^{\alpha_i}\alpha_i!}$$
if $D^k\medpar{f_t,\gamma(k)}$ is non-empty, and zero otherwise.
\end{theorem}
Now, by \cref{MararFormula} but stated for a stable perturbation of an $\eqA$-finite corank $1$ map germ $f:(X,S)\rightarrow(\CC^{n+1},0)$, with $\dim X=n$, we have
\begin{equation}\label{eq: marar f x}
\begin{aligned}
	1+(-1)^n&\mu_I(X,f)=\texttt{\#}S+(-1)^n\mu(X)\\
	&+\sum_{k\geq 2} \sum_{\gamma(k) }a_{\gamma(k)}\left(\beta_0^{\gamma(k)}+(-1)^{\dim D^k\left(f,\gamma(k)\right)} \mu\left(D^k\medpar{f,\gamma(k)}\right)\right),
\end{aligned}\end{equation}
where $\texttt{\#}S$ is the number of points in $S$ and $\beta_0^{\gamma(k)}$ the zero Betti number of $D^k\medpar{f,\gamma(k)}$. 
%

With \cref{LeGreuel} in mind, we take a generic linear projection $p:\CC^{n+1}\rightarrow\CC$, with kernel $H\cong\CC^n$, which defines a transverse slice $g\coloneqq\left.f\right|:\big(X\cap f^{-1}(H),S\big) \rightarrow \CCzero{n}$. 

Observe that $X\cap f^{-1}(H)$, which we will call $\tilde{X}$ to simplify notation, is still an \textsc{icis}. Following the steps of \cite{Nuno-Ballesteros2019} from here, notice that, if $\dim D^k\medpar{f,\gamma(k)}>0$, then $\dim D^k\medpar{f,\gamma(k)}-1=\dim D^k\medpar{g,\gamma(k)}$ and, if $\dim D^k\medpar{f,\gamma(k)}=0$, then the fiber of $D^k\medpar{g,\gamma(k)}$ (i.e., $D^k\big(g_s,\gamma(k)\big)$ for a stable perturbation $g_s$) is empty, by \cref{lem: dkgamma}. Therefore, if we apply the previous formula to $(\tilde{X},g)$, we get
\begin{equation}\label{eq: marar gtildex}
\begin{aligned}
	1+(-1&)^{n-1}\mu_I(\tilde{X},g)=\texttt{\#}S+(-1)^{n-1}\mu(\tilde{X})\\
	&+\,\,\sum_{ \mathclap{\substack{k\geq2, \,\gamma(k):\\ \dim D^k(f,\gamma(k)) >0}} }\ \ \ a_{\gamma(k)}\left(\beta_0^{\gamma(k)}+(-1)^{\dim D^k\left(f,\gamma(k)\right)-1} \mu\left(D^k\medpar{g,\gamma(k)}\right)\right),
\end{aligned}\end{equation}
If we subtract \cref{eq: marar gtildex} from \cref{eq: marar f x}, and use that $\dim D^k\medpar{f,\gamma(k)}=n+1-k-k+\sum_i \alpha_i$ (see \cref{lem: dkgamma}), we have
\begin{align*}
	\mu_I(X,f)+\mu_I&(g,\tilde{X})=\mu(X)+\mu(\tilde{X}) \\
	&+\,\,\sum_{ \mathclap{\substack{k\geq2, \,\gamma(k):\\ \dim D^k(f,\gamma(k)) =0}} }
	\,\,\,\,\,\,\,\,\frac{(-1)^{\sum \alpha_i+1+n}}{\prod_{i\geq1}i^{\alpha_i}\alpha_i!}\left(\beta_0^{\gamma(k)}+\mu\left(D^k\medpar{f,\gamma(k)}\right)\right)\\
	&+\,\,\sum_{ \mathclap{\substack{k\geq2, \,\gamma(k):\\ \dim D^k(f,\gamma(k)) >0}} }
	\,\,\,\,\,\,\,\,\frac{1}{\prod_{i\geq1}i^{\alpha_i}\alpha_i!}\left(\mu\left(D^k\medpar{f,\gamma(k)}\right)+\mu\left(D^k\medpar{g,\gamma(k)}\right)\right),
\end{align*}
where we have simplified the signs expanding $a_{\gamma(k)}$, and $\beta_0^{\gamma(k)}$ denotes the same as before.

Once we arrive here, we can keep simplifying signs: if $\dim D^k(f,\gamma(k))=\sum \alpha_i+1+n-2k=0$, then the first sign is positive. 

On the other hand, we can choose a generic projection and coordinates on source and target so that $p(y_1,\dots,y_{n+1})=y_1$. Moreover, 
\begin{equation*}
D^k\medpar{g,\gamma(k)}=D^k\medpar{f,\gamma(k)}\cap \tilde{p}^{-1}(0), 
\end{equation*}
where $\tilde{p}:X\times\CC^{k-1}\rightarrow \CC$ is the projection on the first coordinate for every $k$, seeing $D^k(f)$ as a subset of $X\times\CC^{k-1}$ (recall the comments below \cref{inmersion}), and it is generic as well (in general it would be a mapping induced by $p\circ f$).

By the comments above, the structure of \textsc{icis} given in \cref{DP} and the Lê-Greuel-type formula for \textsc{icis}; we have
$$ \mu\left(D^k\medpar{f,\gamma(k)}\right)+\mu\left(D^k\medpar{g,\gamma(k)}\right) = \texttt{\#}\Sigma\left(\left.\tilde{p}\right|_{D^k(f_s,\gamma(k))}\right)$$
and
$$ \mu(X)+\mu(\tilde{X})= \texttt{\#}\Sigma\left(\left.\tilde{p}\right|_{X_s}\right),$$
where $f_s$ and $X_s$ are the stable perturbations of $f$ and $X$.

Moreover, note that, if $\dim D^k(f,\gamma(k))=0$, then 
$$ \mu\left(D^k\medpar{f,\gamma(k)}\right)=m_0\left(D^k\medpar{f,\gamma(k)}\right)-\beta_0^{\gamma(k)}, $$
where $m_0\left(D^k\medpar{f,\gamma(k)}\right)$ is the multiplicity of $D^k\big(f,\gamma(k)\big)$. This can also be seen as the number of critical points of $\left.\tilde{p}\right|_{D^k(f_s,\gamma(k))}$.

In conclusion,
	\[\mu_I(X,f)+\mu_I(g,\tilde{X})= \sum_{k\geq 1}\sum_{\gamma(k)}\frac{\texttt{\#}\Sigma\left(\left.\tilde{p}\right|_{D^k(f_s,\gamma(k))}\right)}{\prod_{i\geq1}i^{\alpha_i}\alpha_i!}.
\]

This is exactly the same point Pallarés-Torres and the second author, Nuño-Ballesteros, reach in \cite[Theorem 3.2]{Nuno-Ballesteros2019}. The theorem below follows from there (see \cref{figuralegreuel}).

\begin{theorem}[see {\cite[Theorem 3.2]{Nuno-Ballesteros2019}}]\label{LeGreuelICIS}
For an $\eqA$-finite map germ $f:\left(X,S\right)\rightarrow\left(\CC^{n+1},0\right)$ of corank $1$ from an \textsc{icis} $X$ of dimension $\dim X =n\geq 2$, let $p:\CC^{n+1}\rightarrow\CC$ be a generic linear projection which defines a transverse slice $g:\scalerel*{(}{\strut}X\cap\left(p\circ f\right)^{-1}\left(0\right),S\scalerel*{)}{\strut}\rightarrow \left(\CC^n,0\right)$. Then,
$$ \mu_I\left(f,X\right)+\mu_I\left(g,X\cap \left(p\circ f\right)^{-1}\left(0\right)\right)= \textnormal{\texttt{\#}}\Sigma \left(\left.p\right|_{Z_s}\right),$$
where $\textnormal{\texttt{\#}}\Sigma\left(\left.p\right|_{Z_s}\right)$ is the number of critical points in the stratified sense of $p$ restricted to $Z_s\coloneqq \text{Im}\left(f_s\right)$, being $f_s$ a stable perturbation of $f$.
\end{theorem}

We complete \cref{LeGreuelICIS} with the case of a one-dimensional \textsc{icis} $X$.

\begin{proposition}[see {\cite[Theorem 3.1]{Nuno-Ballesteros2019}}]
Let $f:(X,S)\to(\CC^2,0)$ be an injective map germ from an {\sc icis} $X$ of dimension one. Consider a generic linear projection $p:\CC^2\to\CC$, then
$$\mu_I(X,f)+m_0(f)-1=\textnormal{\texttt{\#}}\Sigma \left(\left.p\right|_{Z_s}\right),$$
where $\textnormal{\texttt{\#}}\Sigma\left(\left.p\right|_{Z_s}\right)$ is the number of critical points in the stratified sense of $p$ restricted to $Z_s\coloneqq \text{Im}\left(f_s\right)$, being $f_s$ a stable perturbation of $f$, and $m_0(f)=\dim_\CC \OO_{X,S}/f^*\mathfrak m_2$ the multiplicity of $f$.
\end{proposition}
%
%
\begin{proof}
We have two strata in the image $Z_s$ of the stable perturbation $(X_s,f_s)$: the 0-dimensional stratum $Z_s^0$ given by the transverse double points and the 1-dimensional stratum $Z_s^1=f_s(X_s)\setminus Z_s^0$. 

Obviously, $\textnormal{\texttt{\#}}\Sigma \medpar{\left.p\right|_{Z_s^0}}$ is equal to $\textnormal{\texttt{\#}}Z_s^0$, which is equal to $\mu_2^\Alt(X,f)$. Since $f_s$ is a local diffeomorphism on $Z_s^1$, $\textnormal{\texttt{\#}}\Sigma \medpar{\left.p\right|_{Z_s^1}}$ is equal to the number of critical points of $p\circ f_s$ on $X_s$ (here the points of $Z_s^0$ can be excluded by genericity of $p$). By the usual Lê-Greuel formula for $X$ and $X\cap (p\circ f)^{-1}(0)$, we have
\[
\textnormal{\texttt{\#}}\Sigma\medpar{p\circ f_s}=\mu(X)+\deg(p\circ f)-1.
\]
But, again, the genericity of $p$ implies that $\deg(p\circ f)=m_0(f)$. Hence,
\begin{align*}
\textnormal{\texttt{\#}}\Sigma \medpar{\left.p\right|_{Z_s^0}}
+\textnormal{\texttt{\#}}\Sigma \medpar{\left.p\right|_{Z_s^1}}&=
\mu_2^\Alt(X,f)+\mu(X)+m_0(f)-1\\&=\mu_I(X,f)+m_0(f)-1,
\end{align*}
by \cref{Houston formula}.
\qed\newline
\end{proof}

\section{The double point Milnor number}\label{Df}

Finding conditions for a 1-parameter family to be Whitney equisingular requires working on the source and in the target separately. In the case of the source, we need to assure some structure and, as the reader could guess, the double point set is the best candidate. Furthermore, this set is the projection of an \textsc{icis}, the double point space, if the map is nice enough (see \cref{fig:DICIS}). We have some invariants in this sense.
%

\begin{definition}\label{mud}
\textit{The double point set of} $f:(\CC^n,S)\to(\CC^{n+1},0)$, of finite singularity type, is the projection on the first coordinate of $D^2(f)$, and we denote it by $D(f)$. Furthermore, if $f$ is $\eqA$-finite, we will define the \textit{double point Milnor number} as
$$\mu_D(f)\coloneqq\beta_{n-1}\medpar{D(f_t)},$$
where $f_t$ is a stable perturbation of $f$.
\end{definition}

\begin{remark}
The double point Milnor number was denoted as $\mu_{\Sigma_2}$ in \cite{Houston2001} by Houston.
\end{remark}

Note that we have to define $\mu_D(f)$ through a stable perturbation of $f$ because $D(f)$ is a hypersurface with possibly non-isolated singularities, hence may not have finite Milnor number in the ordinary sense of a hypersurface. However, we can still use \cref{siersma} to prove that $D(f_t)$ in \cref{mud} has the homotopy type of a wedge of spheres of middle dimension, as a small deformation of $D(f)$ is topologically trivial in a Milnor sphere (see \cref{deftopsiersma}).

The main reason to use this invariant is that $\mu_D(f)$ is the image Milnor number of certain $(X,f)$ if $f$ has corank one, in that case $\mu_D(f)$ coincides with $\mu_I\medpar{D^2(f),\pi}$. Hence, we can use all the machinery we developed above if $f$ has corank one. Firstly, $\mu_D(f)$ is well defined by \cref{muICIS}. Secondly, there are triple points of $f_t$ that also correspond to double points of $\pi$ and give rise to more homology (as \cref{fig:DICIS} represents, there we have depicted a vague idea of the generators of the homology because some higher-dimensional properties cannot be made visible). Finally, the invariant $\mu_D(f)$ is also conservative by \cref{conservationICIS}:

\begin{corollary}
Let $f:\CCS{n}\rightarrow\CCzero{n+1}$ be finitely $\eqA$-determined, of corank $1$, and $f_u$ a one-parameter perturbation of $f=f_0$. Take a representative of the unfolding such that its codomain, $B_\epsilon$, is a Milnor ball.  Then,
$$ \mu_D(f)=\beta_{n-1}\medpar{D(f_u)} +\sum_{y\in B_\epsilon}\mu_D(f_u,y).$$
\end{corollary}

\begin{remark}
Observe that \cref{mud} can be generalized for every pair of dimensions $(n,p)$ as long as $n<p$ and $f$ has corank one. Despite the fact that $D(f)$ could not be a hypersurface, $D^2(f)$ is an \textsc{icis} in this case (see \cref{DICIS}) and we can define $\mu_D(f)$ by means of \cref{homotopydis} (one can also use the Euler-Poincaré characteristic of the image of a stable perturbation, see for example \cite{Nuno-Ballesteros2018}).
\end{remark}

\begin{figure}[htbp]
	\centering
		\includegraphics[width=0.80\textwidth]{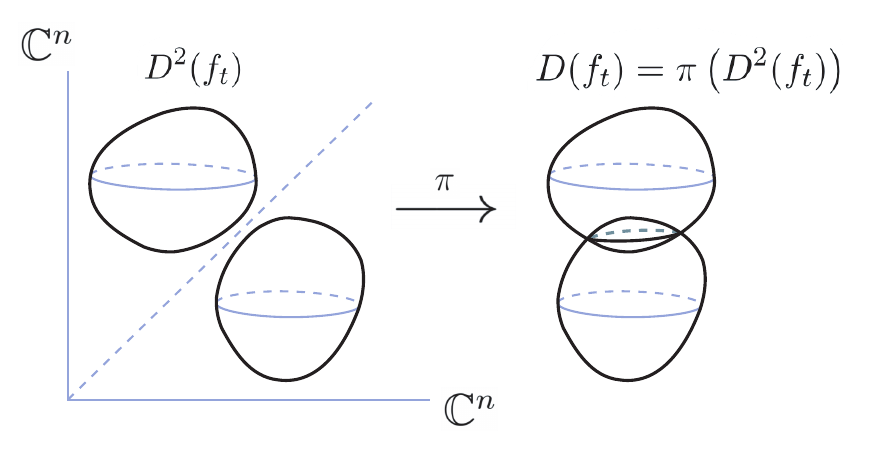}
	\caption{Representation of how the homology of the double point set of a stable perturbation works.}
	\label{fig:DICIS}
\end{figure}


Once more, we focus on the multiple point spaces but, in this case, we deal with $\medpar{D^2(f),\pi}$ and $D^k(\pi)$, where $f$ has corank one and it is $\eqA$-finite. Using the \textit{principle of iteration} (cf. \cite[4.1]{Kleiman1981}), the multiple point spaces of a perturbation of $f$ are isomorphic to the multiple point spaces of a perturbation of $\pi$ with a shift in the multiplicity, and the same is true for unfoldings $F$ and $\Pi$ (of $f$ and $\pi$, respectively). More precisely, $D^k(\pi_t)\cong D^{k+1}(f_t)$ and $D^k(\Pi)\cong D^{k+1}(F)$, where the first isomorphism is given by
\begin{equation}\label{eq: iso principio it}
\begin{aligned}
\phi\colon D^k(\pi_t) &\longrightarrow D^{k+1}(f_t)\\
		\big((x,x_1),\dots,(x,x_k)\big)&\longmapsto(x,x_1,\dots,x_k),
\end{aligned}
\end{equation}
and the second one is analogous. This inspires us to compare $\mu_k^\Alt\scalerel*{(}{\strut}D^2(f),\pi\scalerel*{)}{\strut}$ and $\mu_{k+1}^\Alt(f)$, which determine $\mu_D(f)$ and $\mu_I(f)$, respectively. The relation is straightforward considering that
\begin{align*}
\mu_k^\Alt\medpar{D^2(f),\pi}&=\textnormal{rank }  H^{\Alt_k}_{n-1-k+2}\medpar{D^k(\Pi),D^k(\pi_t)}\\
&=\textnormal{rank } H^{\Alt_k}_{n-1-k+2}\medpar{D^{k+1}(F),D^{k+1}(f_t)}
\end{align*}
and
\begin{equation*}
\mu_{k+1}^\Alt(f)=\textnormal{rank }  H^{\Alt_{k+1}}_{n-1-k+2}\medpar{D^{k+1}(F),D^{k+1}(f_t)}.
\end{equation*}
More precisely, the difference between $\mu^\Alt_k\medpar{D^2(f),\pi}$ and $\mu_{k+1}^{\Alt}(f)$ is the group of permutations that acts. 

To ease the notation, we write $H$ instead of  $H_{n-1-k+2}\big(D^{k+1}(F),D^{k+1}(f_t)\big)$ and $k$ will be clear from the context. Consequently, we want to compare the alternating actions of $\Sigma_{k+1}$ and $\Sigma_k<\Sigma_{k+1}$ on $H$, where $\Sigma_k$  acts as a subgroup fixing the first entry (by construction of the isomorphism of \cref{eq: iso principio it}). We will use representation theory to do this. For this reason, we will see $H$ as a $\CC$-vector space. 

Recall that for each partition of $N$, say $\gamma(N)$, there is associated an irreducible representation of $\Sigma_N$ (see \cite[Proposition 1.10.1]{Sagan2001}). We will call this representation the $\gamma(N)$\textit{-representation}. The representation that acts by its sign is associated to the partition $(1,\dots,1)$, and we call it the \textit{alternating representation}. Moreover, from the branching rules (see \cite[Theorem 2.8.3]{Sagan2001}), we know that the alternating representation of $\Sigma_N$ appears as a restriction of $\Sigma_{N+1}$ from both the alternating representation and the $(2,1,\dots,1)$-representation. Therefore, knowing the character of the last one will be useful. Unfortunately, we could not find it in the literature.

\begin{lemma}\label{caracter}
The character of the irreducible representation associated to the partition $(2,1,\dots,1)$ is
$$ sign(\sigma)\left(\textnormal{fix}\left(\sigma\right)-1\right),$$
where $\textnormal{fix}(\sigma)$ is the number of entries fixed by the permutation $\sigma$.
\end{lemma}
\begin{proof}
If $N=3$, the result is trivial. Assume that $N>3$; then, by \cite[Exercise 4.14]{Fulton1991}, we know that the representation associated to this partition is either the standard or the tensor product of the standard and the alternating representations. By a careful inspection (but see also \cite[Exercise 4.6]{Fulton1991}), we know that it is not the standard representation, therefore its character is the product of the standard and the alternating representations.

\qed\newline

\end{proof}

For any representation $V$ of a finite group $G$ there is a decomposition into irreducible representations where there could be multiple copies of the same irreducible representation, see for example \cite[Proposition 1.8]{Fulton1991}. For any partition $\gamma(N)$ of $N$, the $\gamma(N)$\textit{-isotype}, also written as $V^{\gamma(N)}$, is the sum of those copies of the $\gamma(N)$-representation. In the particular case of the $(1,\dots,1)$-isotype we will use the name \textit{alternating isotype}, as well. Hence, a useful object is the projection onto an isotype.

The projection onto the $(2,1,\dots,1)$-isotype can be made explicit considering we know the character and the dimension of the irreducible representation associated to this partition (see \cite[Section 2.4]{Fulton1991}):
$$ P_{k+1}\coloneqq \frac{k}{(k+1)!}\sum_{\sigma\in \Sigma_{k+1}}sign(\sigma)\left(\textnormal{fix}(\sigma)-1\right)\sigma.$$

\begin{remark}
Observe that one can define the projection $P_{k+1}$ with domain any space where $\Sigma_{k+1}$ acts, here we will define it on $H$ if nothing is said.
\end{remark}

\begin{theorem}\label{comparativak}
Let $f:\CCS{n}\rightarrow\CCzero{n+1}$ be $\eqA$-finite of corank $1$. Then,
\begin{equation*}
\mu_{k+1}^\Alt(f)\leq\mu_k^\Alt\medpar{D^2(f),\pi}
\end{equation*}
for $k=1,\dots,n$. Furthermore,
\begin{enumerate}[\itshape(i)]
	\item for $k=2,\dots,n$, $ \mu_{k+1}^\Alt(f)=\mu_k^\Alt\medpar{D^2(f),\pi} $ if, and only if, $P_{k+1}\equiv0$ (or, equivalently, the $(2,1,\dots,1)$-isotype is zero), and \label{cki}
	\item for $k=1$, $\mu_2^\Alt(f)=\mu\left(D^2(f)\right)$ if, and only if, the space $H_{n-1}\medpar{D^2(f_t)}$ coincides with its alternating isotype, for $f_t$ a stable perturbation of $f$.\label{ckii}
\end{enumerate}
\end{theorem}
\begin{proof}
For $k=2,\dots,n$, from the branching rules, we know that the alternating isotype of the $\Sigma_{k}$ representation on $H$ comes exactly from the alternating isotype and the $(2,1,\dots,1)$-isotype of the representation of $\Sigma_{k+1}$. 

Moreover, the former isotype contributes with the same dimension it has but the latter makes a contribution of one dimension for each $k$-dimensional copy it has. This comes from the fact that, in this isotype, every copy of the $(2,1,\dots,1)$-representation has dimension $k$ and each one splits into an alternating representation of dimension $1$ and an irreducible $(2,1,\dots,1)$-representation of dimension $k-1$ when we restrict it to the subgroup $\Sigma_{k}$. 

As the only difference between $\mu_{k+1}^\Alt(f)$ and $\mu_k^\Alt\medpar{D^2(f),\pi}$ is the different groups acting, $\Sigma_{k+1}$ and $\Sigma_k$ as a subgroup, the result follows for these cases. 

On the other hand, $P_{k+1}\equiv0$ if, and only if, there is no $(2,1,\dots,1)$-isotype, so \cref{cki} follows.

Finally, \cref{ckii} is trivial, as the only two possible representations of $\Sigma_2$ are the trivial and the alternating one.
\qed\newline
\end{proof}

\begin{remark}
It is possible to determine the difference between $\mu_{k+1}^\Alt(f)$ and $\mu_k^\Alt\medpar{D^2(f),\pi}$, for it depends on the number of repetitions of the $(2,1,\dots,1)$-representation. For example, one can compute it through the inner product between the character of the whole representation of $\Sigma_{k+1}$ and the $(2,1,\dots,1)$-representation (see \cite[Corollary 2.16]{Fulton1991}), which is nothing more than counting the number of generators fixed by $P_k$ in a convenient basis.
\end{remark}

Also, for $k=n$ and $k=n+1$, one is dealing with zero homology and the multiple point spaces are points, this ease the relation and we can say something more.

\begin{theorem}\label{comparativan}
With the hypotheses of \cref{comparativak}, for $k=n$,
$$ (n+1) \textnormal{ rank }H_0^{\Alt_{n+1}}\left(D^{n+1}(f_t)\right)=\textnormal{rank } H_0^{\Alt_{n}}\left(D^{n+1}(f_t)\right).$$

Also, for $k=n+1$,
$$\frac{d(f)s(f)^2}{s(f)-1}\mu_{n+2}^\Alt(f)= \mu_{n+1}^\Alt\medpar{D^2(f),\pi}.$$
\end{theorem}
\begin{proof}
Now, we are dealing with points and the zero homology. In particular, we can identify the elements in the homology with the 0-chains. Hence, one can find a basis of $H_0^{\Alt_{n+1}}\left(D^{n+1}(f_t)\right)$ with elements given using the orbits of points $p$, for $p\in D^{n+1}(f_t)$:
$$ \sum_\sigma sign(\sigma)\sigma(p). $$

The action of $\Sigma_{n+1}$ on the homology of the orbit of $p$ is the regular representation, so it decomposes into the alternating representation we are considering, $n$ $(2,1,\dots,1)$-subrepresentations and more irreducible subrepresentations (see \cite[p. 17]{Fulton1991}). The contributions to the alternating isotype of the representation of $\Sigma_n$ come from: one alternating representation from each $(2,1,\dots,1)$-subrepresentation, each one of the alternating subrepresentation of $\Sigma_{n+1}$ will be preserved in the subgroup, and there are no more contributions from other isotypes. This happens for every orbit of points in $D^{n+1}(f_t)$, proving the first statement.



To prove the second part, recall that $\mu_{n+2}^\Alt(f)$ comes from the bottom row of the spectral sequence (see, for example, \cref{1pag}), and the argument is similar but, now, working with the multiple point space of the unfolding. Therefore, again, the alternating isotype of $\Sigma_{k}$ is $k+1$ times bigger than the alternating isotype of $\Sigma_{k+1}$. Hence, if originally $\mu_{n+2}^\Alt(f)$ was
$$ \left|\sum_{\ell=d(f)+1}^{s(f)}(-1)^l {s(f) \choose \ell}\right|={s(f)-1\choose d(f)}, $$
now, $\mu_{n+1}^\Alt\medpar{D^2(f),\pi}$ is
\begin{align*}
	\pushQED{\qed}\left|\sum_{\ell=d(f)+1}^{s(f)}(-1)^\ell \ell{s(f) \choose \ell}\right|&=\frac{d(f)\left(d(f)+1\right)}{s(f)-1}{s(f) \choose d(f)+1}\\
	&=\frac{d(f)s(f)^2}{s(f)-1}{s(f)-1 \choose d(f)}.\qedhere
\popQED
\end{align*}

\end{proof}

\begin{remark}
Although these inequalities and equalities are enough for our purposes, one can specify the relation of $\mu_{n+1}^\Alt(f)$ and $ \mu_{n}^\Alt\medpar{D^2(f),\pi}$ using the ideas of the second part of the proof of \cref{comparativan} and considering the exact sequence of the pair.

Also, one may ask what happens if the group acts by permutation of the elements of a base for some $k<n$ (this action could be not faithful). Regarding this, there are algorithms to compute the alternating part based in the same idea: looking for orbits and the relation between the actions. An upper bound is also possible with the same ideas.
\end{remark}

There are some interesting corollaries of \cref{comparativak,comparativan}. For example, it could happen that there is not enough space in the homology group to fit a $(2,1,\dots,1)$\mdash subrepresentation.

\begin{corollary}
With the notation of \cref{comparativak}, for $k=1,\dots,n$, if 
$$\rank \left(H_{n-1-k+2}\big(D^{k+1}(F),D^{k+1}(f_t)\big)\right)-\mu_{k+1}^\Alt(f)<k,$$ then
\begin{equation*}
\mu_{k+1}^\Alt(f)=\mu_k^\Alt\medpar{D^2(f),\pi}.
\end{equation*}
\end{corollary}
\begin{proof}
The proof is based on the fact that a $(2,1,\dots,1)$-representation of $\Sigma_k$ has dimension $k$ and the ideas of \cref{comparativak,comparativan}.
\qed\newline
\end{proof}

Another example is an inequality involving the full Milnor number on both contexts.

\begin{corollary}\label{ileqd}
For $f$ as in \cref{comparativak}, $\mu_I(f)\leq\mu_D(f)$. This holds with equality if, and only if, $H_{n-1}\medpar{D^2(f_t)}$ coincides with its alternating isotype and all the $P_i$ are zero for all $i$, for $P_i$ as in \cref{comparativak}.
\end{corollary}

Thence, there are some nice characterizations as well, in particular weak  Mond's conjecture for $\mu_D(f)$.

\begin{corollary}[see {\cite[Theorem 3.9]{GimenezConejero2021}}]
For $f$ as in \cref{comparativak}, the following are equivalent:
\begin{enumerate}[\itshape(i)]
	\item $f$ is stable,
	\item $\mu_I(f)=0$, and
		\item $\mu_D(f)=0$.
\end{enumerate}
\end{corollary}
\begin{proof}
If $f$ is stable, then $\mu_D(f)=\mu_I\medpar{D^2(f),\pi}=0$ as so are all the alternating Milnor numbers. If $0=\mu_D(f)$,  then $\mu_I(f)=0$ by \cref{ileqd}, but, if $\mu_I(f)=0$, then  $f$ is stable by the weak Mond's conjecture for corank $1$ (see \cite[Theorem 3.9]{GimenezConejero2021}).
\qed\newline
\end{proof}

For an $f:(\CC^n,S)\to(\CC^{n+1},0)$ of corank one, we have studied the homology of the projection of $D^2(f_t)$ onto $\CC^n$ and we have compared it with the homology of the image of $f_t$, for $f_t$ a stable perturbation of $f$. One can keep looking for relations between the multiple point spaces using the same ideas. 

On one hand we can reproduce the roles of the image of $f_t$ and $D(f_t)$ easily:
$$
 \begin{tikzcd}
 \cdots\arrow[dr,dashed, two heads ] \arrow[r,dashed]&   D^3(f_t) \arrow[dr, two heads] \arrow[r, "\pi^3_2" ]& D^2(f_t) \arrow[dr, two heads] \arrow[r, "\pi^2_1" ]& D^1(f_t)=\CC^n \arrow[dr, two heads] \arrow[r, "f_t" ]&\CC^{n+1}\\
 \cdots& D_3^4(f_t)\arrow[u, hook] &D_2^3(f_t)\arrow[u, hook] &D_1^2(f_t)=D(f_t)\arrow[u, hook] &\textnormal{Im}(f_t)\arrow[u, hook]
\end{tikzcd},
$$

where $\pi_{k-1}^k:D^{k}(\bullet)\to D^{k-1}(\bullet)$ is the projection that forgets the last entry, for $k=2,\dots,d(f)$, and $D_{k-1}^k(\bullet)$ is the image of $\pi_{k-1}^k$.

On the other hand, although $\pi_{k-1}^k:D^k(f)\rightarrow D^{k-1}(f)$ has the problem that the target is an \textsc{icis} as well, the homology of $D_{k-1}^k(f_t)$ is well defined for any $\eqA$-finite $f:\CCS{n}\rightarrow\CCzero{n+1}$ of corank $1$. This is a consequence of \cref{SecuenciaH} applied to $\pi_{k-1}^k:D^k(f_t)\rightarrow D^{k-1}(f_t)$. Hence, again by the iteration principle, the Betti numbers of $D_{k-1}^k(f_t)$ are determined by the $\Sigma_{i+1}$-alternated homologies of $\left(D^{k+i}(F),D^{k+i}(f_t)\right)$, for $i=0,\dots,d(f)-k$. 

Furthermore, $D^{k+i}(f_t)$ is the unique fiber, up to isomorphism, of $D^{k+i}(f)$ and the action of the permutations does not depend on the stable perturbation, so these Betti numbers will be well defined. 

Finally, note that the homology will appear again in middle dimension for the pair of dimensions $(n,n+1)$, for the same reason it happens for $D(f_t)$. Hence, we will simply write $\beta_k(f)$ to denote $\beta_{n-k+1}\scalerel*{(}{\strut}D_{k-1}^k(f_t)\scalerel*{)}{\strut}$.
\newline


We have compared $\beta_1(f)\coloneqq\mu_I(f)$ with $\beta_2(f)=\mu_D(f)$ and, similarly, we can compare $\beta_k(f)$ with $\beta_{k+1}(f)$. This is very easy if $d(f)<n+1$ or we have a mono-germ, because we can forget about the the homology of the pair and the unfolding by the exact sequence of the pair. As we were saying, by \cref{SecuenciaH} and the iteration principle, arranging the terms in a convenient way, we have
\begin{equation}\label{eq: bk+1}
	\beta_{k+1}(f)=\rank\bigoplus_{i\geq2}H^{\Alt_i}\medpar{D^{k+i}(f_t)}\oplus H\medpar{D^{k+1}(f_t)}
	\end{equation}
	and
	\begin{equation}\label{eq: bk}
	\beta_k(f)=\rank\bigoplus_{i\geq1}H^{\Alt_{i+1}}\medpar{D^{k+i}(f_t)}\oplus H\medpar{D^k(f_t)},
	\end{equation}
from where we have omitted the index of the homology. Therefore, subtracting \cref{eq: bk} from \cref{eq: bk+1} and using the ideas of \cref{comparativak,comparativan}, we get
\begin{equation}
\begin{aligned}
\beta_{k+1}(f)-\beta_k(f)=&\sum_{i\geq2} \frac{\textnormal{rank }H\medpar{D^{k+i}(f_t)}^{(2,1,\dots,1)}}{i}+\mu\medpar{D^{k+1}(f)}\\
&-\mu\medpar{D^k(f)}-\textnormal{rank } H^{\Alt_2}\medpar{D^{k+1}(f_t)}
\label{eq:betis}
\end{aligned},
\end{equation}
with $(2,1,\dots,1)$ partition of $i+1$.

\begin{remark}
One can take this to a broader context as long as the first page of the spectral sequence collapses and the iteration principle works.
\end{remark}

We put this in practice with some examples.

\begin{example} \label{2to3k1}
For $f:\Gzero{2}{3}$ as in \cref{comparativak} and taking $k=1$ in \cref{eq:betis},
\begin{align*}
 \mu_D(f)-\mu_I(f)&= \frac{\textnormal{ rank }H_0\left(D^3(f_t)\right)^{(2,1)}}{2}+\mu\left(D^2(f)\right)-0-\mu_2^\Alt(f)\\
&=\frac{\textnormal{ rank }H_0\left(D^3(f_t)\right)^{(2,1)}}{2}+\rank H_1\left(D^2(f_t)\right)^{(2)},
 \end{align*}
where $H_1\big(D^2(f_t)\big)^{(2)}$ is the part of the homology that is fixed by the group $\Sigma_2$, i.e., the \textit{trivial isotype} of $\Sigma_2$. The last equality is due to the fact that there are only two irreducible representations of $\Sigma_2$, the alternating and the \textit{trivial} one.
\end{example}

\begin{example} \label{2to3k2}
For $f:\Gzero{2}{3}$ as in \cref{comparativak}, and taking $k=2$ in \cref{eq:betis},
\begin{align*}
 \beta_3(f)-\mu_D(f)&= \textnormal{rank }H_0\left(D^3(f_t)\right)-\mu\left(D^2(f)\right)-\textnormal{rank }H_0^{\Alt_2}\left(D^3(f_t)\right)\\
&=-\mu\left(D^2(f)\right)+ \textnormal{rank } H_0\left(D^3(f_t)\right)^{(2)},
 \end{align*}
following the same notation as above.
\end{example}

Note that the triple points of $f_t$ are strict in \cref{2to3k1,2to3k2}, i.e.,  in $D^3(f_t)$ the points are the $\Sigma_3$-orbit of $(a,b,c)$ with $a\neq b\neq c\neq a$. Say we have $T$ triple points, then 
\begin{itemize}
	\item  $\textnormal{ rank }H_0\left(D^3(f_t)\right)^{(2,1)}=4T$, as it is the complement of the alternating and trivial isotype and both representations have dimension one.
	\item $\textnormal{rank } H_0\left(D^3(f_t)\right)^{(2)}=3T$, as it is the trivial isotype of $\Sigma_2$ fixing the first entry (it has elements of the form $(a,b,c)+(a,c,b)$).
	\item $\rank H_0^\Alt\left(D^3(f_t)\right)=3T$, similarly as the previous case.
	\item $\beta_3(f)=6T$, as it is simply counting the elements of the orbits.
\end{itemize}
In conclusion, we have the following results, which were obtained also by Houston with similar invariants (see \cite[Proofs of Theorems 2.7 and 2.8]{Houston2001}).
\begin{theorem}
For $f:\Gzero{2}{3}$ $\eqA$-finite and being $T$ the number of triple points of a stable perturbation of $f$,
$$\mu_D(f)=\mu\left(D^2(f)\right)+3T.$$
\end{theorem}
\begin{proof}
If $f$ has corank 1, we are done by \cref{Houston formula}. If $f$ does not have corank one, $D^2$ has dimension one and $D^3$ is zero dimensional, so the homology is in middle dimension and the same argument can be applied.

\qed\newline
\end{proof}

Similarly:

\begin{theorem}
For $f:\Gzero{3}{4}$ as in \cref{comparativak} and being $Q$ the number of quadruple points of a stable perturbation $f_t$ of $f$,
$$\mu_D(f)=4Q+\mu\medpar{D^2(f)}+\frac{\mu\medpar{D^3(f)}-\mu_3^T(f)+\mu_3^\Alt(f)}{2},$$
where $\mu_3^T(f)\coloneqq\textnormal{rank } H_1\medpar{D^3(f_t)}^{(3)}$ is the invariant homology by $\Sigma_3$ and $\mu_3^\Alt(f)$ is defined as in \cref{mukalticis}, i.e., the alternating homology by $\Sigma_3$.
\end{theorem}
\begin{proof}
By \cref{eq:betis} for $k=1$, we have that
\begin{align*}
\mu_D(f)-\mu_I(f)=&\frac{\textnormal{rank }H\medpar{D^{3}(f_t)}^{(2,1)}}{2}+\frac{\textnormal{rank }H\medpar{D^{4}(f_t)}^{(2,1,1)}}{3}\\
&+\mu\medpar{D^2(f)}-0-\mu_2^\Alt(f).
\end{align*}
Now, observe that $\Sigma_3$ has only three irreducible representations: the trivial representation, alternating representation and the $(2,1)$-representation. With this in mind, 
$$\frac{\textnormal{rank }H\medpar{D^{3}(f_t)}^{(2,1)}}{2}=\frac{\textnormal{rank }H\medpar{D^{3}(f_t)}-\mu_3^T(f)-\mu_3^\Alt(f)}{2}.$$

Finally, recall that $\mu_I(f)=\mu_2^\Alt(f)+\mu_3^\Alt(f)+\mu_4^\Alt(f)$ (see \cite[Definition 3.11]{Houston2010} or \cref{Houston formula}). Furthermore,
$$ \frac{\textnormal{rank }H\medpar{D^{4}(f_t)}^{(2,1,1)}}{3}=3Q $$
and
$$ \mu_4^\Alt(f)=Q.$$
The result follows from here.
\qed\newline
\end{proof}

There is another relation between $\mu_I$ and $\mu_D$.
Let $f:(\CC^n,S)\to(\CC^{n+1},0)$ be $\eqA$-finite of corank one and let $g:(\CC^{n-1},S')\to(\CC^{n},0)$ be a transverse slice. By the Lê-Greuel type formula (recall \cref{LeGreuel,LeGreuelICIS}), we know that
\[
\mu_I(f)+\mu_I(g)=\textnormal{\texttt{\#}}\Sigma(p|_{Z_s})=\sum_{\mathcal Q}\textnormal{\texttt{\#}}\Sigma\big(p|_{\mathcal Q(f_s)}\big),
\]
where $p:\CC^{n+1}\to\CC$ is the generic projection which defines the transverse slice, $Z_s$ is the image of a stable perturbation of $f$, $\mathcal Q$ runs through all the stable types in the target and $\mathcal{Q}(f_s)$ denotes the points of $f_s$ in the target that are of stable type $\mathcal{Q}$.

Using the same argument, we have that
\[
\mu_D(f)+\mu_D(g)=\textnormal{\texttt{\#}}\Sigma(p\circ f_s)=\sum_{\mathcal Q_S}\textnormal{\texttt{\#}}\Sigma\big((p\circ f_s)|_{\mathcal Q_S(f_s)}\big),
\]
where now $\mathcal Q_S$ runs through all the stable types in the source.

If a stable type in the source $\mathcal Q_S$ corresponds to $\mathcal Q$ in the target, the restriction $f_s:\mathcal Q_S(f_s)\to\mathcal Q(f_s)$ is a local diffeomorphism and is $r$-to-one, where $r=r(\mathcal Q)$ is the number of branches of the stable type $\mathcal Q$. Hence,
\[
\textnormal{\texttt{\#}}\Sigma((p\circ f_s)|_{\mathcal Q_S(f_s)})=r(\mathcal Q)\textnormal{\texttt{\#}}\Sigma(p|_{\mathcal Q(f_s)}).
\]

Therefore, if $\mu_I(f_t)$ and $\mu_I(g_t)$ are constant in a family, this implies that $\mu_D(f_t)$ and $\mu_D(g_t)$ are also constant, by upper semi-continuity.

\begin{proposition}\label{prop: quitar mud}
Let $f:(\CC^n,S)\to(\CC^{n+1},0)$ be $\eqA$-finite of corank one and let $g:(\CC^{n-1},S')\to(\CC^{n},0)$ be a transverse slice. If $\mu_I(f_t)$ and $\mu_I(g_t)$ are constant in a family, then $\mu_D(f_t)$ and $\mu_D(g_t)$ are also constant.
\end{proposition}

The previous argument could make the reader think that, when we deal with Whitney equisingularity, controlling the target is enough to control the source, or vice versa. This idea is wrong in general. We take care of the details to control the target and the source in \cref{sec: we}, and the problem with the previous idea is that $\mu_D(g)$ is not what we need (see \cref{LeGreuelM,equisingularidad}). See also \cref{ex: source no suficiente} for an example where this idea fails.


\section{Whitney equisingularity}\label{sec: we}

In \cite{Gaffney1993}, Gaffney showed that a  one parameter family $f_t:\GS{n}{n+1}$ is Whitney equisingular if, and only if, it is excellent (in Gaffney's sense) and all the polar multiplicities in the source and target are constant on $t$. The problem is that, for each $d$-dimensional stratum in the source or target, we need $d+1$ invariants, so the total number of invariants we need to control the Whitney equisingularity is huge. 

In this section, we follow the approach of Teissier in \cite{Teissier1982} for hypersurfaces with isolated singularities or Gaffney in \cite{Gaffney1996} for {\sc icis} to show that, in the corank one case, Whitney equisingularity can be characterized in terms of the $\mu_I^*(f_t)$ and $\mu_I^*\medpar{D^2(f_t),\pi}$ sequences, obtained by taking successive transverse slices of $f_t$.
\newline

In \cref{sec:Le-Greuel}, we already made use of the stratification $\mathscr S$ by stable types of the image of a locally stable mapping $f:X\to Y$ between smooth manifolds $X$ and $Y$ with $\dim X=n$ and $\dim Y=n+1$ and with only corank one singularities. Since each stable type is determined by its Mather algebra $\mathcal{Q}$, we can denote by $\mathcal{Q}(f)$ the stratum of points $y\in f(X)$ such that the multi-germ of $f$ at $y$ has type $\mathcal{Q}$. 
Because $f$ is stable, $\mathscr S$ is a partial stratification of $f$ in the sense of \cite[Proposition 3.1]{Gibson1976a}. It follows that we have an induced stratification $\mathscr S'$ on $X$, with strata $\mathcal{Q}_S(f)=f^{-1}\medpar{\mathcal{Q}(f)}$, such that
 $f:X\to Y$ is a Thom stratified map.

Suppose, now, that we have an $\eqA$-finite germ $f:(X,S)\rightarrow(\CC^{n+1},0)$ of corank one, where $X$ is an $n$-dimensional {\sc icis}. By \cref{MGcriterion}, we can take a finite representative $f:X\to Y$, where $Y$ is an open neighbourhood of 0 in $\CC^{n+1}$ such that $f^{-1}(0)=S$ and $f:X\setminus S\to Y\setminus\{0\}$ is a locally stable mapping. The stratification by stable types on $f:X\setminus S\to Y\setminus\{0\}$ extends to $f:X\to Y$ just by adding $S$ and $\{0\}$ as strata in the source and target, respectively. By shrinking the representative if necessary, we can always assume that $f$ has no $0$-stable singularities, so $S$ and $\{0\}$ are in fact the only 0-dimensional strata.

Finally, we give a version of excellency, stratification by stable types and Whitney equisingularity for unfoldings of germs on \textsc{icis}.

Let $(\mathcal{X},\pi,F,j)$ be a one parameter unfolding of $(X,f)$ which is origin preserving (that is, $S\subset X_t$ and $f_t(S)=0$, for all $t$) so we can see the unfolding as a family of germs $f_t:(X_t,S)\to(\CC^{n+1},0)$.

\begin{definition} We say that $(\mathcal{X},\pi,F,j)$ is \emph{excellent} if there exist a representative $F:\mathcal X\to Y\times U$, where $Y$ and $T$ are open neighbourhoods of the origin in $\CC^{n+1}$ and $\CC$ respectively, such that for all $t\in T$, $f_t^{-1}(0)=S$ and $f_t:X_t\setminus S\to Y\setminus\{0\}$ is a locally stable mapping with no $0$-stable singularities.
\end{definition}

When the unfolding is excellent, $F:\mathcal X\setminus S\times\{0\}\to (Y\setminus\{0\})\times T$ is also stable, so we have a well defined stratification by stable types. This extends to $F:\mathcal X\to Y\times T$ just by adding $S\times T$ and $\{0\}\times T$ as strata in the source and target, respectively. These are, in fact, the only 1-dimensional strata.

\begin{definition}\label{def:WE} We say that $(\mathcal{X},\pi,F,j)$ is \emph{Whitney equisingular} if $F:\mathcal X\to Y\times T$ is a Thom stratified map with the stratification by stable types.
\end{definition}

Now, we recall the definition of polar multiplicities, following Gaffney in \cite{Gaffney1993}.

\begin{definition}
Let $f:(X,S)\rightarrow(\CC^{n+1},0)$ be $\eqA$-finite of corank one. For each stable type $\mathcal Q$ such that $d=\dim \mathcal Q(f)>0$ and for each $i=0,\dots,d-1$, the \emph{$i$th-polar multiplicities} in the source and target are
\[
m_i(f,\mathcal Q)=m_0\left(P_i\medpar{\overline{\mathcal Q(f)}}\right),\quad m_i(f,\mathcal Q_S)=m_0\left(P_i\medpar{\overline{\mathcal Q_S(f)}}\right),
\]
where the bar means the Zariski closure and $P_i(Z)$ is the absolute polar variety of codimension $i$ of $Z$ in the sense of Lê and Teissier in \cite[p. 462]{Trang1981}.

The \emph{$d$th-stable multiplicities} are
\begin{align*}
m_d(f,\mathcal Q)&=\deg\left(\pi:P_d\medpar{\overline{\mathcal Q(F)},\pi}\to(\CC^r,0)\right),\\
m_d(f,\mathcal Q_S)&=\deg\left(\pi:P_d\medpar{\overline{\mathcal Q_S(F)},\pi}\to(\CC^r,0)\right),
\end{align*}
where, now, $(\mathcal{X},\pi,F,j)$ is an $r$-parameter versal unfolding of $(X,f)$ and $P_d(\mathcal Z,\pi)$ is the relative polar variety of codimension $d$ of a family $\pi:\mathcal Z\to\CC^r$ (see \cite[Section IV.1]{Teissier1982}).

Finally, we denote by $c(f)$ the number of all $0$-stable singularities that appear in a stable perturbation of $(X,f)$.
\end{definition}

It follows from the definition of relative polar variety that the top polar multiplicity $m_d(f,\mathcal Q)$ is equal to the number of critical points of $p|_{\mathcal Q(f_s)}$, where $p:\CC^{n+1}\to\CC$ is a generic linear projection and $(X_s,f_s)$ is a stable perturbation of $(X,f)$. Since the $0$-stable singularities are also critical points of $p$ in the stratified sense, we get the following reformulation of \cref{LeGreuel}:

\begin{corollary}\label{LeGreuelM}
For a corank $1$ and $\eqA$-finite multi-germ $f:(X,S)\rightarrow (\CC^{n+1},0)$, $X$ an \textsc{icis} of dimension $\dim X=n>1$,  let $p:\CC^{n+1}\rightarrow\CC$ be a generic linear projection which defines a transverse slice $g:(Y,S)\rightarrow (\CC^n,0)$, where $Y=X\cap\left(p\circ f\right)^{-1}\left(0\right)$. Then,
$$\mu_I(X,f)+\mu_I(Y,g)=\sum_{\textnormal{dim}\mathcal{Q}(f_s)=d>0}m_d(f,\mathcal{Q})+c(f),$$
where the sum runs on all $\mathcal{Q}$ such that $\textnormal{dim}\mathcal{Q}(f_s)=d$ and all $d>0$.
\end{corollary}

Now, we define the $\mu_I^*$ and $\mu_D^*$-sequences of a corank one map germ.

\begin{definition}\label{sequences}
Consider $f:(X,S)\rightarrow (\CC^{n+1},0)$ $\eqA$-finite of corank one, with $X$ an \textsc{icis} of dimension $\dim X=n>1$. We take a generic flag of vector subspaces
\[
H_{(n-1)}\subset\dots\subset H_{(1)}\subset H_{(0)}=\CC^{n+1},
\]
such that $H_{(i)}$ has codimension $i$. We put $X_{(i)}=X\cap f^{-1}\left(H_{(i)}\right)$ and $f_{(i)}=\left.f\right|_{X_{(i)}}$
and define the $\mu_I^*$\textit{-sequence of $(X,f)$} as
$$\mu_I^*(X,f)\coloneqq\left(\mu_I(X,f),\mu_I\left(X_{(1)},f_{(1)},\right),\dots,\mu_I\left(X_{(n-1)},f_{(n-1)}\right)\right).$$
Sometimes we do not consider the top image Milnor number $\mu_I(X,f)$ in the $\mu_I^*$-sequence and, then, we denote it by $\tilde\mu_I^*(X,f)$.
\end{definition}

It is well-known that, by generic, we mean a suitable Zariski open in a convenient space, and this definition does not depend on the generic flag we are taking. The details can be seen in \cite[pp. 1380--1381]{Marar2014}.
\newline

In next lemma, we see that all polar multiplicities can be seen as top polar multiplicities of the corresponding transverse slices.

\begin{lemma}\label{lem:polar} With the hypothesis and notation of \cref{sequences}, suppose that $\mathcal Q$ is a stable type such that $\dim\mathcal Q(f_s)=d>0$, for a stable perturbation $(X_s,f_s)$ of $(X,f)$. Then,
\[
m_{d-i}\left(f_{(i)},\mathcal{Q}\right)=m_{d-i}\left(f,\mathcal{Q}\right),\ i=1,\dots,d.
\]
\end{lemma}

\begin{proof}
By induction, it is enough to prove that 
\[m_{d-i}\left(f_{(1)},\mathcal{Q}\right)=m_{d-i}\left(f,\mathcal{Q}\right),\ i=1,\dots,d.
\] 
To see this, we first observe that $\mathcal Q\medpar{f_{(1)}}=\mathcal Q(f)\cap H_{(1)}$, so the equality for $i=2,\dots,d$ follows directly from \cite[Corollary 4.1.9]{Trang1981}.

For $i=1$, we can see $f$ as a stabilisation of $f_{(1)}$. 
If $\ell:\CC^{n+1}\to\CC$ is the linear form such that $H_{(1)}=\ell^{-1}(0)$, this means that $f|_{\ell^{-1}(t)}$, with $t\ne0$, is a stable perturbation of $f_{(1)}$.
In particular, $m_{d-1}\medpar{f_{(1)},\mathcal{Q}}$ is  the number of critical points of a generic linear projection $p:\CC^{n+1}\to\CC$ restricted to $\mathcal Q\medpar{f|_{\ell^{-1}(t)}}=\mathcal Q(f)\cap \ell^{-1}(t)$. This number can be also seen as 
\begin{equation}\label{degd-1}
\deg\left(\ell: P_{d-1}\medpar{\overline{\mathcal Q(f)},\ell}\to(\CC,0)\right),
\end{equation}
where $P_{d-1}\medpar{\overline{\mathcal Q(f)},\ell}$ is the closure of the set of critical points of $(p,\ell)|_{\mathcal Q(f)}$.

On the other hand, $m_{d-i}\left(f,\mathcal{Q}\right)=m_0\medpar{P_{d-1}\medpar{\overline{\mathcal Q(f)}}}$. Since $P_{d-1}\medpar{\overline{\mathcal Q(f)}}$ is 1-dimensional and $\ell$ is generic, this is equal to
\begin{equation}\label{degd-2}
\deg\left(\ell: P_{d-1}\medpar{\overline{\mathcal Q(f)}}\to(\CC,0)\right),
\end{equation}
where $P_{d-1}\medpar{\overline{\mathcal Q(f)}}$ is again the closure of the set of critical points of $(p,\ell)|_{\mathcal Q(f)}$. So, \cref{degd-1,degd-2} are equal.
\qed
\newline
\end{proof}

We arrive to the main theorem, which characterises the Whitney equisingularity of a family of map germs in terms of the $\mu_I^*$-sequences of $f_t$ and $\medpar{D^2(f_t),\pi}$.

\begin{theorem}\label{equisingularidad}
Let $f_t\colon(\CC^n,S)\rightarrow (\CC^{n+1},0)$ be a one parameter family of $\eqA$-finite corank one map germs. Then, the family is Whitney equisingular if, and only if, the sequences $\mu_I^*(f_t)$ and $\tilde{\mu}_I^*\medpar{D^2(f_t),\pi}$ are constant on $t$.
\end{theorem}

\begin{proof}
Suppose that the sequences $\mu_I^*(f_t)$ and $\tilde{\mu}_I^*\medpar{D^2(f_t),\pi}$ are constant.
First of all, the constancy of $\mu_I(f_t)$ implies that the family is excellent (see \cite[Theorem 4.3]{GimenezConejero2021}). 
In fact, this holds not only for the family $f_t$, but also for all families ${f_t}_{(i)}$, $i=1,\dots,n-1$. In particular, all the numbers of 0-stable singularities $c\medpar{{f_t}_{(i)}}$ are constant (see \cite[Proposition 3.6]{Gaffney1993}).
By Gaffney's results in \cite[Theorems 7.1 and 7.3]{Gaffney1993}, we need to proof that all polar invariants in the source and target are constant.

By \cref{lem:polar}, the constancy of the polar multiplicities follows from the constancy of the top polar multiplicities of all the transverse slices ${f_t}_{(i)}$, with $i=1,\dots,n-1$. Secondly, we apply recursively \cref{LeGreuelM} on ${f_t}_{(i)}$ for $i=0,\dots,n-2$. For any $i$, we  have
$$\mu_I\medpar{{f_t}_{(i)}}+\mu_I\medpar{{f_t}_{(i+1)}}=\sum_{\textnormal{dim}\mathcal{Q}({f_t}_{(i)})=d>0}m_d\medpar{{f_t}_{(i)},\mathcal{Q}}+c\medpar{{f_t}_{\left(i\right)}}.$$
The polar multiplicities $m_d\medpar{{f_t}_{\left(i\right)},\mathcal{Q}}$ are upper semi-continuous (see \cite[Proposition 4.15]{Gaffney1993}). Therefore, all $m_d\medpar{{f_t}_{(i)},\mathcal{Q}}$ must be constant.

For the polar multiplicities in the source, we follow the same argument, but this time applied to the family $\medpar{D^2(f_t),\pi}$. Observe that the polar multiplicities of ${f_t}_{(i)}$ in the source coincide with the polar multiplicities of $\medpar{D^2({f_t}_{(i)}),\pi}$ in the target. Hence, we need to study $\mu_I^*\medpar{D^2({f_t}_{(i)}),\pi}$.
Moreover, it follows from \cref{prop: quitar mud} that
\[
\mu_D(f_t)+\mu_D\medpar{{f_t}_{(1)}}=\sum_{\mathcal Q_S} r(\mathcal Q)\textnormal{\texttt{\#}}\Sigma\medpar{p|_{\mathcal Q(f_s)}}.
\]
Since the right-hand side is constant, for all the members in the sum are either polar multiplicities in the target or numbers of $0$-stable invariants, so is the left-hand side. Again, the upper semi-continuity of $\mu_D$ (see \cref{uscontinuity}) implies that $\mu_D(f_t)$ is also constant. Hence, it is enough to consider the reduced sequence $\tilde{\mu}^*_I\medpar{D^2(f_t),\pi}$, as $\mu_D(f_t)=\mu_I\medpar{D^2(f_t),\pi}$.

Finally, the converse is easy. If the family $f_t$ is Whitney equisingular, so are the families ${f_t}_{(i)}$, for $i=1,\dots,n-1$. By Thom's second isotopy lemma, they are topologically trivial and, hence, their image Milnor numbers are constant (see  \cite[Corollary 2.10]{GimenezConejero2021}). Analogously, the family $\medpar{D^2(f_t),\pi}$ must be also Whitney equisingular, which gives the constancy of the sequence $\tilde{\mu}^*_I\medpar{D^2(f_t),\pi}$.
\qed
\newline
\end{proof}

The following example shows that $\tilde{\mu}^*_I\medpar{D^2(f_t),\pi}$, or even $\mu^*_I\medpar{D^2(f_t),\pi}$, is not enough to control Whitney equisingularity.

\begin{example}\label{ex: source no suficiente}
Consider the one parameter family $f_t:(\CC^2,0)\to(\CC^3,0)$ such that $$f_t(x,y)=(x,y^4,x^5y-5x^3y^3+4xy^5+y^6+ty^7).$$
This family is an example of topologically trivial family such that it is not Whitney equisingular, which was shown in \cite[Example 5.5]{Ruas2019}. 

In this example, we also have constancy of $\mu\medpar{D(f_t)}$. Therefore, the source is Whitney regular as $D(f_t)$ is a family of plane curves. In particular, all the multiplicities in the source and $\mu_I^*\medpar{D^2(f_t),\pi}$ are constant. However, neither the polar multiplicities in the target nor $\mu_I^*(f_t)$ are constant because the family is not Whitney equisingular (and by \cref{equisingularidad}).
\end{example}

%
%
%

The proof of \cref{equisingularidad} allows us to state a partial result when only the sequence $\mu_I^*(f_t)$ is constant. 

\begin{definition}
We say that the family $f_t$ is \textit{Whitney equisingular in the target} if there exists a representative of the unfolding $F:\mathcal{X}\to Y\times T$ as in \cref{def:WE} such that the stratification by stable types on $Y\times T$ is a Whitney stratification. 
\end{definition}

\begin{proposition}\label{prop: eq source}
Let $f_t\colon(\CC^n,S)\rightarrow (\CC^{n+1},0)$ be a one-parameter family of $\eqA$-finite corank one map germs. Then, the family is Whitney equisingular in the target if, and only if, the sequence $\mu_I^*(f_t)$ is constant on $t$.
\end{proposition}

\begin{example}
Related to \cref{ex: source no suficiente}, there are families $f_t$ that are Whitney equisingular in the target, but they are not Whitney equisingular. In \cite{NunoBallesteros2022}, Tomazella, Silva, and the second author prove that the family $f_t:(\CC^3,0)\to(\CC^4,0)$ such that $$f_t(x,y,z)=(x,y,z^5+xz,z^7+yz+tz^8)$$ has
\begin{align*}
	\mu_I^*(f_t)&=(25,16,10), \forall t,\\
	\tilde{\mu}_I\big(D^2(f_0),\pi\big)&=(65,9), \text{ for }t=0,\\\
	\tilde{\mu}_I\big(D^2(f_t),\pi\big)&=(65,8), \text{ for }t\neq0.
\end{align*}
This shows that the family is Whitney equisingular only in the target (by \cref{equisingularidad,prop: eq source}). This and \cref{ex: source no suficiente} enhance the importance of all the invariants of \cref{equisingularidad}.
\end{example}

As a corollary of \cref{equisingularidad}, and closing the topic we have started with \cref{ex: source no suficiente}, the Whitney equisingularity of a family $f_t:(\CC^2,S)\to(\CC^3,0)$ is controlled by just two invariants in the target. In fact, in \cite{Marar2014} it was shown that $f_t$ (of any corank) is Whitney equisingular if, and only if, $\mu\medpar{D(f_t)}$ and $\mu\medpar{\im({f_t}_{(1)})}$ are constant, where $\mu$ is the usual Milnor number of a plane curve.

\begin{corollary} Let $f_t\colon(\CC^2,S)\rightarrow (\CC^{3},0)$ be a one parameter family of $\eqA$-finite corank one map germs. Then, the family is Whitney equisingular if, and only if, $\mu_I(f_t)$ and $\mu_I\medpar{{f_t}_{(1)}}$ are constant on $t$.
\end{corollary}

\bibliographystyle{plain}
\bibliography{Mybib}

\begin{thebibliography}{10}

\bibitem{CisnerosMolina2019}
Jos\'e~Luis Cisneros-Molina and David Mond.
\newblock Multiple points of a simplicial map and image-computing spectral
  sequences.
\newblock {\em arXiv:1911.11095}, November 2019.

\bibitem{Damon1991a}
James Damon.
\newblock {${\mathscr{A}}$}-equivalence and the equivalence of sections of
  images and discriminants.
\newblock In {\em Singularity theory and its applications, {P}art {I}
  ({C}oventry, 1988/1989)}, volume 1462 of {\em Lecture Notes in Math.}, pages
  93--121. Springer, Berlin, 1991.

\bibitem{Damon1991}
James Damon and David Mond.
\newblock {$\mathscr{A}$}-codimension and the vanishing topology of
  discriminants.
\newblock {\em Inventiones Mathematicae}, 106(2):217--242, 1991.

\bibitem{Fulton1991}
William Fulton and Joe Harris.
\newblock {\em Representation theory}, volume 129 of {\em Graduate Texts in
  Mathematics}.
\newblock Springer-Verlag, New York, 1991.
\newblock A first course, Readings in Mathematics.

\bibitem{Gaffney1993}
Terence Gaffney.
\newblock Polar multiplicities and equisingularity of map germs.
\newblock {\em Topology. An International Journal of Mathematics},
  32(1):185--223, 1993.

\bibitem{Gaffney1996}
Terence Gaffney.
\newblock Multiplicities and equisingularity of {ICIS} germs.
\newblock {\em Invent. Math.}, 123(2):209--220, 1996.

\bibitem{Gibson1976a}
Christopher~G. Gibson, Klaus Wirthm\"{u}ller, Andrew~A. du~Plessis, and Eduard
  J.~N. Looijenga.
\newblock {\em Topological stability of smooth mappings}.
\newblock Lecture Notes in Mathematics, Vol. 552. Springer-Verlag, Berlin-New
  York, 1976.

\bibitem{GimenezConejero2021}
R~Gim\'enez~Conejero and J~J Nu{\~{n}}o-Ballesteros.
\newblock {The Image Milnor Number and Excellent Unfoldings}.
\newblock {\em The Quarterly Journal of Mathematics}, 03 2021.
\newblock haab019.

\bibitem{Robertothesis}
Roberto Gim{\'e}nez~Conejero.
\newblock {\em Singularities of germs and vanishing homology}.
\newblock PhD thesis, Universitat de Val{\`e}ncia, 2021.

\bibitem{Goryunov1993}
Victor Goryunov and David Mond.
\newblock Vanishing cohomology of singularities of mappings.
\newblock {\em Compositio Mathematica}, 89(1):45--80, 1993.

\bibitem{Goryunov1995}
Victor~V. Goryunov.
\newblock Semi-simplicial resolutions and homology of images and discriminants
  of mappings.
\newblock {\em Proceedings of the London Mathematical Society. Third Series},
  70(2):363--385, 1995.

\bibitem{Greuel1975}
G.~M. Greuel.
\newblock Der {G}auss-{M}anin-{Z}usammenhang isolierter {S}ingularit\"{a}ten
  von vollst\"{a}ndigen {D}urchschnitten.
\newblock {\em Mathematische Annalen}, 214:235--266, 1975.

\bibitem{Ament2017}
Daiane~Alice Henrique~Ament and Juan Jos\'{e} Nu\~{n}o Ballesteros.
\newblock Mond's conjecture for maps between curves.
\newblock {\em Math. Nachr.}, 290(17-18):2845--2857, 2017.

\bibitem{Houston1997}
Kevin Houston.
\newblock Local topology of images of finite complex analytic maps.
\newblock {\em Topology. An International Journal of Mathematics},
  36(5):1077--1121, 1997.

\bibitem{Houston2001}
Kevin Houston.
\newblock Calculating generalised image and discriminant {M}ilnor numbers in
  low dimensions.
\newblock {\em Glasgow Mathematical Journal}, 43(2):165--175, 2001.

\bibitem{Houston2007}
Kevin Houston.
\newblock A general image computing spectral sequence.
\newblock In {\em Singularity theory}, pages 651--675. World Sci. Publ.,
  Hackensack, NJ, 2007.

\bibitem{Houston2010}
Kevin Houston.
\newblock Stratification of unfoldings of corank 1 singularities.
\newblock {\em The Quarterly Journal of Mathematics}, 61(4):413--435, 2010.

\bibitem{Houston2011}
Kevin Houston.
\newblock Equisingularity of families of hypersurfaces and applications to
  mappings.
\newblock {\em Michigan Math. J.}, 60(2):289--312, 2011.

\bibitem{JorgeSaia2006}
V.~H. Jorge~P\'{e}rez and M.~J. Saia.
\newblock Euler obstruction, polar multiplicities and equisingularity of map
  germs in {$\mathscr O(n,p),\ n<p$}.
\newblock {\em Internat. J. Math.}, 17(8):887--903, 2006.

\bibitem{Kleiman1981}
Steven~L. Kleiman.
\newblock Multiple-point formulas. {I}. {I}teration.
\newblock {\em Acta Mathematica}, 147(1-2):13--49, 1981.

\bibitem{Liu2021}
Yongqiang Liu, Guillermo Pe{\~{n}}afort-Sanchis, and Matthias Zach.
\newblock Cohomological connectivity of perturbations of map-germs.
\newblock {\em arXiv:2103.14685}, 2021.

\bibitem{Marar1991}
W.~L. Marar.
\newblock The {E}uler characteristic of the disentanglement of the image of a
  corank {$1$} map germ.
\newblock In {\em Singularity theory and its applications, {P}art {I}
  ({C}oventry, 1988/1989)}, volume 1462 of {\em Lecture Notes in Math.}, pages
  212--220. Springer, Berlin, 1991.

\bibitem{Marar2014}
W.~L. Marar and J.~J. Nu\~{n}o Ballesteros.
\newblock Slicing corank 1 map germs from {$\Bbb C^2$} to {$\Bbb C^3$}.
\newblock {\em The Quarterly Journal of Mathematics}, 65(4):1375--1395, 2014.

\bibitem{Marar1989}
Washington~Luiz Marar and David Mond.
\newblock Multiple point schemes for corank {$1$} maps.
\newblock {\em Journal of the London Mathematical Society. Second Series},
  39(3):553--567, 1989.

\bibitem{McCleary2001}
John McCleary.
\newblock {\em A user's guide to spectral sequences}, volume~58 of {\em
  Cambridge Studies in Advanced Mathematics}.
\newblock Cambridge University Press, Cambridge, second edition, 2001.

\bibitem{Mond1987}
David Mond.
\newblock Some remarks on the geometry and classification of germs of maps from
  surfaces to {$3$}-space.
\newblock {\em Topology}, 26(3):361--383, 1987.

\bibitem{Mond1991}
David Mond.
\newblock Vanishing cycles for analytic maps.
\newblock In {\em Singularity theory and its applications, {P}art {I}
  ({C}oventry, 1988/1989)}, volume 1462 of {\em Lecture Notes in Math.}, pages
  221--234. Springer, Berlin, 1991.

\bibitem{Mond1995}
David Mond.
\newblock Looking at bent wires---{${\mathscr A}_e$}-codimension and the
  vanishing topology of parametrized curve singularities.
\newblock {\em Math. Proc. Cambridge Philos. Soc.}, 117(2):213--222, 1995.

\bibitem{Mond1994}
David Mond and James Montaldi.
\newblock Deformations of maps on complete intersections, {D}amon's
  {$\mathcal{K}_V$}-equivalence and bifurcations.
\newblock In {\em Singularities ({L}ille, 1991)}, volume 201 of {\em London
  Math. Soc. Lecture Note Ser.}, pages 263--284. Cambridge Univ. Press,
  Cambridge, 1994.

\bibitem{Mond-Nuno2020}
David Mond and J.~J. Nu{\~{n}}o-Ballesteros.
\newblock {\em Singularities of mappings}, volume 357 of {\em Grundlehren der
  mathematischen Wissenschaften}.
\newblock Springer, Cham, 2020.

\bibitem{Nuno-Ballesteros2018}
J.~J. Nu{\~{n}}o-Ballesteros, B.~Or\'{e}fice-Okamoto, and J.~N. Tomazella.
\newblock Equisingularity of families of isolated determinantal singularities.
\newblock {\em Mathematische Zeitschrift}, 289(3-4):1409--1425, 2018.

\bibitem{Nuno-Ballesteros2019}
J.~J. Nu{\~{n}}o-Ballesteros and I.~Pallar\'{e}s-Torres.
\newblock A {L}\^{e}-{G}reuel type formula for the image {M}ilnor number.
\newblock {\em Hokkaido Mathematical Journal}, 48(1):45--59, 2019.

\bibitem{Nuno2017}
Juan~J. Nu{\~{n}}o-Ballesteros and Guillermo Pe{\~{n}}afort-Sanchis.
\newblock Multiple point spaces of finite holomorphic maps.
\newblock {\em Q. J. Math.}, 68(2):369--390, 2017.

\bibitem{Robinson1996}
Derek J.~S. Robinson.
\newblock {\em A course in the theory of groups}, volume~80 of {\em Graduate
  Texts in Mathematics}.
\newblock Springer-Verlag, New York, second edition, 1996.

\bibitem{Ruas2019}
M.~A.~S. Ruas and O.~N. Silva.
\newblock Whitney equisingularity of families of surfaces in {$\Bbb C^3$}.
\newblock {\em Mathematical Proceedings of the Cambridge Philosophical
  Society}, 166(2):353--369, 2019.

\bibitem{Sagan2001}
Bruce~E. Sagan.
\newblock {\em The symmetric group}, volume 203 of {\em Graduate Texts in
  Mathematics}.
\newblock Springer-Verlag, New York, second edition, 2001.
\newblock Representations, combinatorial algorithms, and symmetric functions.

\bibitem{Siersma1991}
Dirk Siersma.
\newblock Vanishing cycles and special fibres.
\newblock In {\em Singularity theory and its applications, {P}art {I}
  ({C}oventry, 1988/1989)}, volume 1462 of {\em Lecture Notes in Math.}, pages
  292--301. Springer, Berlin, 1991.

\bibitem{Teissier1982}
Bernard Teissier.
\newblock Vari\'{e}t\'{e}s polaires. {II}. {M}ultiplicit\'{e}s polaires,
  sections planes, et conditions de {W}hitney.
\newblock In {\em Algebraic geometry ({L}a {R}\'{a}bida, 1981)}, volume 961 of
  {\em Lecture Notes in Math.}, pages 314--491. Springer, Berlin, 1982.

\bibitem{Trang1974a}
L\^{e}~D{\~{u}}ng Tr\'{a}ng.
\newblock Computation of the {M}ilnor number of an isolated singularity of a
  complete intersection.
\newblock {\em Funkcional. Anal. i Prilo{\v{z}}en 8(2)}, 8(2):45--49, 1974.

\bibitem{Trang1987}
L\^{e}~D{\~{u}}ng Tr\'{a}ng.
\newblock Le concept de singularit\'{e} isol\'{e}e de fonction analytique.
\newblock In {\em Complex analytic singularities}, volume~8 of {\em Adv. Stud.
  Pure Math.}, pages 215--227. North-Holland, Amsterdam, 1987.

\bibitem{Trang1992}
L\^{e}~D{\~{u}}ng Tr\'{a}ng.
\newblock Complex analytic functions with isolated singularities.
\newblock {\em Journal of Algebraic Geometry}, 1(1):83--99, 1992.

\bibitem{Trang1981}
L\^{e}~D{\~{u}}ng Tr\'{a}ng and Bernard Teissier.
\newblock Vari\'{e}t\'{e}s polaires locales et classes de {C}hern des
  vari\'{e}t\'{e}s singuli\`eres.
\newblock {\em Annals of Mathematics. Second Series}, 114(3):457--491, 1981.

\bibitem{Wall1981}
C.~T.~C. Wall.
\newblock Finite determinacy of smooth map-germs.
\newblock {\em The Bulletin of the London Mathematical Society},
  13(6):481--539, 1981.

\end{thebibliography}
\end{document}